\documentclass{amsart}
\usepackage{fullpage}
\usepackage{graphicx}
\usepackage{color}
\newtheorem{theorem}{Theorem}[section]
\newtheorem{corollary}[theorem]{Corollary}
\newtheorem{proposition}[theorem]{Proposition}
\newtheorem{assump}[theorem]{Assumption}
\newtheorem{problem}[theorem]{Problem}
\usepackage{hyperref}
\providecommand{\doi}[1]{Eprint \href{http://dx.doi.org/#1}{doi:#1}}

\catcode`\@=11\def\th@mydefinition{%
  \thm@notefont{\bfseries}}
\catcode`\@=12
\RequirePackage{amssymb}
\renewcommand{\geq}{\geqslant}
\renewcommand{\leq}{\leqslant}

\theoremstyle{mydefinition}
\newtheorem{example}[theorem]{Example}
\theoremstyle{definition}
\newtheorem{definition}[theorem]{Definition}
\newtheorem{remark}[theorem]{Remark}

\newtheorem{lemma}[theorem]{Lemma}
\newcommand{\mnmatrices}{\cM_{\nrows,\ncols}(\rmax)}
\newcommand{\llq}{\text{``}}
\newcommand{\rrq}{\text{''}}
\newcommand{\adj}{^{\mathrm{adj}}}

\DeclareMathAlphabet{\mathbbold}{U}{bbold}{m}{n}
\newcommand{\zero}{\mathbbold{0}}
\newcommand{\unit}{\mathbbold{1}}

\newcommand{\nrows}{m}
\newcommand{\ncols}{n}
\newcommand{\cw}{\operatorname{cw}}
\newcommand{\set}[2]{\{#1;\,#2\}}
\newcommand{\res}{^\sharp}
\newcommand{\balance}{\,\nabla\, }
\newcommand{\notbalance}{\,\not\!\nabla\, }
\newcommand{\rmax}{\mathbb{R}_{\max}}
\newcommand{\Ti}{{\mathbb T}_{\mathrm e}}

\newcommand{\R}{\mathbb{R}}
\newcommand{\Rm}{\R\cup\{-\infty\}}
\newcommand{\N}{\mathbb{N}}
\newcommand{\proj}{\pi}
\newcommand{\inj}[1]{#1^\vee}
\newcommand{\NP}{\text{\sc NP}}
\newcommand{\coNP}{\text{\sc co-NP}}
\newcommand{\NEW}[1]{{\em #1\/}\index{#1}}
\newcommand{\cM}{\mathcal{M}}
\newcommand{\detp}[1]{\per #1}
\newcommand{\per}[1]{\operatorname{per}#1}
\def\allperm{\mathfrak{S}}

\title[Tropical polyhedra and mean payoff games]{Tropical polyhedra are equivalent to mean payoff games}
\author{Marianne Akian}
\address{Marianne Akian,
INRIA Saclay--\^Ile-de-France and CMAP, \'Ecole 
Polytechnique. Address:
CMAP, \'Ecole Polytechnique,
Route de Saclay,
91128 Palaiseau Cedex, France.}
\email{Marianne.Akian@inria.fr}
\author{St\'ephane Gaubert}
\address{St\'ephane Gaubert,
INRIA Saclay--\^Ile-de-France and CMAP, \'Ecole 
Polytechnique. Address: CMAP, \'Ecole Polytechnique,
Route de Saclay,
91128 Palaiseau Cedex, France. Phone +33 1 69 33 46 13, Fax +33 1 69 33 46 46}
\email{Stephane.Gaubert@inria.fr}
\author{Alexander Guterman}
\address{Alexander Guterman, Moscow State University, Leninskie Gory, 119991, GSP-1,
Moscow, Russia}
\email{guterman@list.ru}
\thanks{The two first authors were partially supported by the joint RFBR-CNRS grant 05-01-02807, by a MSRI Research membership for the Fall 2009 Semester on Tropical Geometry, and by a grant from LEA (Laboratoire Europ{\'e}en Associ{\'e}) MathMode. The second author was also partially supported by the Arpege  programme of the French National Agency of Research (ANR), project ``ASOPT'', number ANR-08-SEGI-005 and by the Digiteo project DIM08 ``PASO'' number 3389.
The third author was partially supported by the invited professors program from INRIA Paris-Rocquencourt and by the grants MD-2535.2009.1 and RFBR 09-01-00303.}
\catcode`\@=11
\@namedef{subjclassname@2010}{%
  \textup{2010} Mathematics Subject Classification}
\catcode`\@=12

\date{December 4, 2009; Revised: April 25, 2011; June 9, 2011.}
\begin{document}
\maketitle

\begin{abstract}
We show that several decision problems originating from
max-plus or tropical convexity are equivalent to
zero-sum two player game problems.
In particular, we set up an equivalence between
the external representation of tropical convex sets
and zero-sum stochastic games, in which tropical polyhedra
correspond to deterministic games with finite action spaces. 
Then, we show that the 
winning initial positions can be determined from
the associated tropical polyhedron.
We obtain as a
corollary a game theoretical proof
of the fact that the tropical rank of a matrix, defined as
the maximal size of a submatrix for which the optimal assignment
problem has a unique solution, coincides with the maximal
number of rows (or columns) of the matrix which are linearly independent
in the tropical sense. Our proofs rely on techniques from non-linear Perron-Frobenius theory.
\end{abstract}
\section{Introduction}
\subsection{Statement of the problems and main results}
The three following problems are basic in max-plus or tropical algebra.

\begin{problem}[Is a tropical polyhedral cone non-trivial?]\label{pb-1}
Given $m\times n$ matrices $A=(A_{ij})$ and $B=(B_{ij})$ with entries
in $\Rm$, does there exist a vector $x\in (\Rm)^n$ non-identically $-\infty$
such that the inequality $\llq Ax\leq Bx\rrq$ holds in the tropical sense, i.e.,
\begin{align}
\max_{j\in [n]} \bigl( A_{ij}+ x_j \bigr) \leq \max_{j\in[n]} \bigl(  B_{ij}+ x_j \bigr) , \qquad
\forall i\in [m] \enspace ?
\label{e-fgcone}
\end{align}
\end{problem}
Here and in the sequel, we use the notation $[n]:=\{1,\ldots,n\}$. 
\begin{problem}[Is a tropical polyhedron empty?]\label{pb-2}
Given $m\times n$ matrices $A=(A_{ij})$ and $B=(B_{ij})$ with entries
in $\Rm$, and two vectors $c,d$ of dimension $m$ with entries in $\Rm$,
does there exist a vector $x\in (\Rm)^n$ 
such that the inequality $\llq Ax+c\leq Bx+d\rrq$ holds in the tropical sense, i.e.,
\begin{align}\label{e-affinepol}
\max\big(\max_{j\in [n]} (A_{ij}+ x_j),c_i\big) \leq \max\big(\max_{j\in[n]} (B_{ij}+ x_j),d_i\big), \qquad
\forall i\in [m] \enspace ?
\end{align}
\end{problem}

\begin{problem}[Is a family of vectors tropically dependent?]\label{pb-indep}
Given $m\geq n$ and an $m \times n$ matrix $A=(A_{ij})$ with entries in $\Rm$, 
are the columns of $A$ tropically linearly dependent? I.e., can
we find scalars $x_1,\ldots,x_n\in \R\cup\{-\infty\}$, not all
  equal to $-\infty$, such that the equation $\llq Ax=0\rrq$ holds in the tropical sense, meaning that for every value of $i\in[m]$,
when evaluating the expression
\[
\max_{j\in[n]} (A_{ij}+x_j) 
\]
the maximum is attained by at least two values of $j$?
\end{problem}

The representation of a tropical polyhedral cone by inequalities
turns out to be equivalent to the description of a mean payoff
game by a bipartite directed graph in which the weights
indicate the payments (the weighted graph is coded by the matrices
$A$ and $B$). More generally, we consider an infinite system
of inequalities, the set $[m]$ being replaced by an infinite set in~\eqref{e-fgcone}. The set $P$ of solutions of this system is now a tropical convex cone
(not necessarily polyhedral), and we associate to it a mean payoff
game with an infinite set of actions, corresponding to defining
half-spaces. We shall see that such infinite systems of inequalities
represent in particular {\em stochastic} mean payoff (zero-sum) games.

Our main results set up a correspondence between the external
representation (by inequalities) of a tropical convex cone
$P$, and mean payoff games with infinite action spaces on one side, in which
\[
\exists u\in P, \; u \not\equiv -\infty \Leftrightarrow \text{ there is at least one winning initial state.}
\]
Moreover, when $P$ is polyhedral, the actions spaces becomes finite,
and
\[
\exists u\in P, \; u_i \neq -\infty \Leftrightarrow \quad i \text{ is a winning initial state} %
,
\]
see Theorems~\ref{theo-reduce} and~\ref{theo-reduce2}.
This shows that Problem~\ref{pb-1} and its affine version, Problem~\ref{pb-2},
are (polynomial-time) equivalent to mean payoff game problems. We show by the same techniques that Problem~\ref{pb-indep} reduces to a mean payoff
game problem, and derive theoretical results concerning
tropical linear dependence by game techniques.

Before discussing further these results, we give more background. 
Note that
some of the present results were announced in the conference paper~\cite{AkianGaubertGuterman10}, without proofs.
\subsection{Motivation}
The first two problems concern max-plus or tropical {\em convex sets}. The latter
are subsets $C$ of $(\R\cup\{-\infty\})^n$ such that
\[
u,v\in C,\qquad \lambda,\mu\in \R\cup\{-\infty\},\qquad
\max(\lambda,\mu)=0 \implies (\lambda +u)\vee (\mu+v) \in C
\]
where ``$\vee$'' is the supremum operator for the partial order of 
$\R\cup\{-\infty\}$, that is the ``max'' applied entrywise, and where $\lambda+u$
denotes the vector obtained by adding the scalar $\lambda$ to every entry
of $u$. 

Max-plus or tropical convexity has been developed by several researchers
under different names. It goes back at least to the work of
Zimmermann~\cite{zimmerman77}. It was studied by Litvinov, Maslov, and 
Shpiz~\cite{litvinov00}, in relation to problems of calculus of variations,
and by Cohen, Gaubert, and Quadrat~\cite{cgq00,cgq02}, motivated by discrete event system problems~\cite{ccggq99}. Max-plus polyhedra have appeared in tropical geometry after the
work of Develin and Sturmfels~\cite{DS04}, followed by several works including
the ones of Joswig and Yu, see~\cite{joswig04,JSY07}. Recent works
on the subject include~\cite{BriecHorvath04,blockyu06,GK,BSS,NiticaSinger07a,joswig-2008,katz08,GM08,AlGK09}.

As it is shown in~\cite{SGKatz} (see also~\cite{maxplus97,katz08})
max-plus polyhedra can be defined equivalently
in terms of generators (extreme points and rays) or relations
(linear or affine inequalities).
In particular, a max-plus polyhedral cone can be defined by
systems of the form $\llq Ax\leq  Bx\rrq$, whereas max-plus polyhedra
can be defined by their affine analogues, $\llq Ax+c\leq Bx+d\rrq$.
Max-plus polyhedra have been used in particular in~\cite{katz05,LGKL09} to solve
controllability and observability problems for discrete event systems,
and they have been used in~\cite{XAGG08} as a new domain in static analysis
by abstract interpretation, allowing one to express disjunctive constraints.
The question of solving $\llq Ax\leq Bx\rrq$ over (finite) relative integers
has also been considered in~\cite{bezem2,bezem3} with motivations from SMT (SAT-modulo theory) solving.

In many applications, it is necessary to pass from the description
by inequalities $\llq Ax\leq Bx\rrq$ to the description by extreme
generators. Although a reasonably efficient hypergraph based algorithm has been
developed~\cite{AGG10double}, its applicability is limited by the exponential
blowup of the number of extreme generators, in the worst case~\cite{AlGK09}.
The alternative approach of tropical geometry~\cite{DS04,joswig-2008}, in
which a tropical polyhedron is represented by a classical polyhedral
complex, is subject to a more severe exponential blowup (the number of cells of the complex exceeds the one of tropical extreme generators). 
However, in several applications, including 
the final control synthesis or observer synthesis step in~\cite{katz05,LGKL09},  one only needs to find a single solution or to decide that 
there is none. This subproblem, which is expected to be much simpler,
is the object of Problems~\ref{pb-1} and~\ref{pb-2}, the latter being the affine analogue
of the former.

The third problem, concerning linear dependence, is motivated by tropical
geometry. In this setting, the {\em tropical hyperplane}~\cite{RGST}
determined by a vector $u\in (\R\cup\{-\infty\})^n$, 
non identically $-\infty$,
is defined as the set of points $x\in (\R\cup\{-\infty\})^n$ such that the maximum
in the expression
\[
\max_{i\in [n]}(u_i +x_i) 
\]
is attained at least twice, which may be written as $\llq u\cdot x=0\rrq$,
tropically. This arises naturally when considering
amoebas, which are the images of algebraic varieties over a valued field by the map which takes the valuation entrywise. We refer the reader
to the survey by Itenberg, Mikhalkin and Shustin~\cite{itenberg} 
for more background. In particular, one
may consider the field $\mathbb{K}:=\mathbb{C}\{\{t\}\}$ of formal Puiseux series over $\mathbb{C}$
in an indeterminate $t$, equipped with 
the non-archimedean valuation $v$ which associates to a series
the opposite of the smallest exponent of the monomials appearing
in the series. Then, tropical hyperplanes can be seen to be amoebas
of hyperplanes over $\mathbb{K}$ (this is a special case of a general result
of Kapranov characterizing the non-archimedean amoebas of hypersurfaces, see~\cite{kapranov}).
Moreover, if the columns of a matrix
with entries in $\mathbb{K}$ are linearly dependent over $\mathbb{K}$,
their images by the map which does the valuation entrywise 
can be seen to be linear
dependent in the tropical sense, i.e., in the sense used in Problem~\ref{pb-indep}, see~\cite{RGST}, and so, deciding the linear dependence turns out
to be a basic issue in tropical linear algebra. 

\subsection{Discussion of the result}
One interest of the transformation to mean payoff games that we describe here is of an algorithmic nature. Mean payoff games have been well
studied since the work of Gurvich, Karzanov, and Khachiyan~\cite{gurvich2},
who developed the first combinatorial algorithm to solve them.
Since that time, the existence of a polynomial time algorithm
has been an open question. Recall in this respect
that decision problems concerning the value of mean payoff games
are known to be in $\NP\cap \coNP$, see Condon~\cite{condon2} and
Zwick and Paterson~\cite{zwick} for more information.

Pseudo-polynomial algorithms of value iteration type are known~\cite{zwick},
and other types algorithms have been developed~\cite{gg,dhingra,bjorklund,JPZSIAM}. These include policy iteration algorithms, which are experimentally fast
on typical inputs, although Friedmann showed recently~\cite{Friedmann-AnExponentialLowerB} that a commonly used class of policy improvement
rules leads to a worst case exponential execution
time.
The present transformations allow one to apply
any mean payoff game algorithm to solve Problems~\ref{pb-1}--\ref{pb-indep}.

The problem of finding one solution of the system $\llq Ax\leq Bx\rrq$ was previously considered without the connection with games
by Butkovi\v{c} and Zimmerman, who developed a combinatorial algorithm~\cite{MR2203194}. The latter turns out to be only pseudo-polynomial, as shown by Bezem, Nieuwenhuis and Rodr\'\i guez-Carbonell~\cite{bezem}.
It is conceivable to extend the method of~\cite{MR2203194} to solve the affine case (Problem~\ref{pb-2}) in pseudo-polynomial time. We are not aware of methods preexisting to the present work allowing one to solve Problem~\ref{pb-indep} 
in pseudo-polynomial time.

If one requires the vector $x$ to be finite, Problem~\ref{pb-1} becomes simpler.
In this special case, a reduction which inspired the present one was made
by Dhingra and Gaubert, who showed in~\cite[\S~IV,C]{dhingra} (Corollary~\ref{theo-reduce3} below) that 
$\llq Ax\leq Bx\rrq$ has a finite solution if and only if all
the initial states of an associated mean payoff game are winning.
A related result was established previously by 
Mohring, Skutella and Stork~\cite{skutella}, who 
studied a scheduling problem with and/or precedence constraints,
leading to a feasibility problem which is equivalent to finding
a finite vector in a tropical polyhedron. They showed that the
latter problem is polynomial time equivalent to deciding
whether a mean payoff game has a winning state. 
In~\cite{skutella} as well as in~\cite{dhingra} and the present
work, a mean payoff game is canonically associated (by a
syntaxic construction) to the feasibility problem. 
Then, the approach of~\cite{skutella} requires an additional
transformation, adding some auxiliary states, with special
weights (determined by a value iteration argument).

The present work relies on a different approach, based on non-linear
Perron-Frobenius theory. This allows us to deal at the same
time with infinite coordinates (i.e., tropically zero coordinates,
this is an essential
matter, both in applications and for theoretical
reasons) and with infinite systems
of inequalities, i.e., with tropical convex sets instead
of tropical polyhedra, the former representing
games with infinite action spaces on one side, including
stochastic games. 
Even in the case of finite coordinates and finite systems,
this approach leads to simpler results,
since there is no need to transform
the game as in~\cite{skutella}. 
In particular, our first result shows that the system $\llq Ax\leq Bx\rrq$ has a tropically nonzero solution (possibly with infinite entries) if and only if the associated game
has at least one winning initial state. 
In the case of Problem~\ref{pb-1}, a key ingredient of the proof
is a non-linear extension due to Nussbaum~\cite{MR818894} of the Collatz-Wielandt characterization of the spectral radius of a 
matrix with nonnegative entries. 
Our approach %
to Problem~\ref{pb-2} relies
on Kohlberg's theorem and is therefore in the more special setting
of polyhedra.

Thus, the results of the present paper allow one to apply game theory
algorithms to solve problems of tropical algebra, but conversely, they
also allow one to transfer results from tropical algebra to game theory.
In particular, the set of elements of a max-plus (or min-plus) polyhedron coincides with the set of ``bias'' or ``potential'' vectors which
are used classically to certify that 
the value of a mean payoff game is nonnegative (or nonpositive).
Precise structural results on tropical polyhedra are available 
(including a description by extreme points and rays, see~\cite{agg10,AlGK09}
and the references therein), and so, our results
yield an explicit representation of the set of potentials. 

A related reduction, albeit of a different nature,
was recently pointed out by Schewe~\cite{schewe}, 
who showed that 
solving a mean payoff game reduces to a feasibility problem
in linear programming but with exponentially large coefficients.
The latter reduction appears to be related to the ``dequantization'' method in tropical geometry
(see in particular the proof of Theorem~1 in~\cite{AlGK09}).

Finally, we note that after the submission of this paper,
the present results and ideas have been applied in two further works: 
in~\cite{agg10},
it is shown that the tropical analogue of 
Farkas lemma, i.e., checking whether a tropical
linear inequality can be logically deduced from
a finite family of tropical linear inequalities,
is also equivalent to a mean payoff game problem,
whereas in~\cite{gks10}, a reduction of tropical linear
programming to parametric mean payoff games is presented.
\subsection{A theorem concerning the tropical rank}\label{sec1.4}
It is natural
to look for characterizations of tropical linear independence
in terms of determinants. The tropical analogue of the determinant
of an $n \times n$ matrix $B$ (with entries in $\R\cup\{-\infty\}$)
is the value of the optimal assignment problem
\begin{align}\label{e-def-per}
\max_{\sigma} \left(\sum_{i\in[n]} B_{i\sigma(i)} \right)
\end{align}
where the maximum is taken over all the permutations $\sigma$ of the set $[n]$.
Following Develin, Santos, and Sturmfels~\cite{DSS}, we say that a matrix $B$ with entries in $\R$ is {\em tropically singular} if the above maximum is attained by at least two permutations.
The same notion was first considered by Butkovi\v{c} in~\cite{butkovip94,But} (tropically nonsingular matrices being qualified there of {\em strong regular} matrices). We shall indeed use the following extension of the above definition to the case of matrices $B$ with entries in  $\R\cup\{-\infty\}$: $B$ is {\em tropically singular} if the above maximum is either attained by at least 
two permutations, or equal to $-\infty$.
As a corollary of our game reduction
of Problem~\ref{pb-indep}, we obtain the following result (see Theorem~\ref{th-main}). 
\begin{theorem}\label{theorem2}
Let $A$ be an $m \times n$ matrix with entries in $\R\cup\{-\infty\}$,
with $m\geq n$. Then, the columns of $A$ are tropically linearly independent
if and only if $A$ has a tropically non-singular $n\times n$ submatrix.
\end{theorem}
This was first stated by Izhakian, and proved in the square case,
in~\cite{Izh}. The proof of the rectangular case given in the
present paper, relying on mean payoff games, was announced 
in~\cite{AGG08}. Meanwhile, Izhakian and Rowen
completed the proof in the rectangular case~\cite{IRowen}, using
a different approach. 
In fact, as shown in~\cite{AGG08}, the ``if'' part of the result can be
deduced from the max-plus Cramer theory~\cite{Plus} (see also~\cite{RGST,AGG08}),
and the square case is related to a result of Gondran and Minoux~\cite{gondran78}.
It should also be noted that the
special case of Theorem~\ref{theorem2} in which the entries of the matrix $A$ are finite
can be derived alternatively from a result of Develin, Santos, and Sturmfels~\cite[Theorem~5.5]{DSS}, showing that the Kapranov rank of a matrix is maximal if and only if its tropical rank is maximal. 
However, the present game approach 
(Theorem~\ref{th:4equivmaxplus} below) yields a pseudo-polynomial algorithm and implies that the corresponding decision problem (checking whether the tropical rank is maximal) is in $\NP\cap \coNP$.

What is surprising is that Theorem~\ref{theorem2} still holds in the rectangular case, since
the analogue of this result in the ``signed'' case, in which tropical hyperplanes are replaced by sets of the form
\[
H=\{x\in (\R\cup\{-\infty\})^n\mid \max_{i\in I} (u_i+x_i) = \max_{i\in J} (u_i+x_i) \}
\]
where $I$ and $J$ are disjoint non-empty subsets of $[n]$, and the definition
of tropical singularity is modified accordingly, turns out to be non valid
in the rectangular case, as shown by the counter example of~\cite{AGG08}. 
This shows that some tropical linear algebra issues
are better behaved when thinking of max-plus numbers as images
by the valuation of complex Puiseux series rather than real ones.

It is instructive to compare the different proofs
of Theorem~\ref{theorem2} in the special case of matrices
with finite entries: the one of~\cite{DSS} relies on
a mixed subdivision technique (the ``Cayley trick'') combined
with Sperner's coloring lemma; the one of Izhakian and Rowen~\cite{IRowen}, 
uses a reduction to the square case by a direct inductive argument;
the ``game'' proof that we present here relies on some fixed
point type results (Kohlberg's theorem or Nussbaum's Collatz-Wielandt
property). There turns out to be a fourth proof which we include
in the final section of this paper, in which the rectangular
case is reduced to the square case by a direct application of the tropical Helly theorem. Whereas the arguments
in~\cite{Plus,Izh} made in the case of square matrices can be interpreted
in terms of network flows, it should be noted that a general result of Sturmfels
and Zelevinsky~\cite{SZ} showing that the
Newton polytope of the product of maximal minors of the general $m\times n$
matrix is not a transportation polytope, unless $|m-n|\leq 1$, indicates
that the direct network flow approach is unlikely to carry over to the rectangular case.

Finally, let us point out that Kim and Rousch~\cite{kim} showed
that computing the tropical rank is an $\NP$-hard problem,
whereas our results suggest that the subproblem of
checking whether a matrix is of maximal tropical rank might be easier,
because it belongs to $\NP\cap \coNP$. In addition, it can
be solved in pseudo-polynomial time. More generally, our result
implies that for a fixed $k$, checking whether an $m \times n$ matrix has
tropical rank at least $n-k$ reduces to solving a polynomial number
(namely ${n \choose k}$) of
mean payoff game problems (Corollary~\ref{cor-trivial}).
Since checking whether the tropical rank is at most $k$ is 
a polynomial time problem (Remark~\ref{rk-trivial}),
this suggests that the difficulty of the problem
of computing the tropical rank may be concentrated in instances in which
the tropical rank $k$ is such that $1\ll k\ll \min(m,n)$.

\typeout{TODO. Perhaps discuss the equivalent form of the results in terms of the tropical plucker coordinates (cf. tropical grassmanian of speyer).}
\section{Preliminary results}
\subsection{Mean payoff games arising from tropical cones}\label{sec-gen}
We first recall some basic definitions 
concerning mean payoff games and associate a mean payoff
game to a tropical cone. The state space of the game
will turn out to be finite precisely when the cone is polyhedral.
The max-plus semiring $\rmax$ is the set of real numbers, 
completed by $-\infty$, equipped with the
addition $(a,b)\mapsto \max(a,b)$ and the multiplication $(a,b)\mapsto a+b$.
The name ``tropical'' will be used in the sequel as a synonym of ``max-plus''.

The reader is referred to~\cite{cgq02,DS04} for more background
on max-plus or tropical convexity, and in particular 
to~\cite{cgqs04,SGKatz} for the results on external representation.

A tropical {\em closed convex cone} can be defined externally
by a system of linear tropical inequalities of the form           
\begin{align}\label{sys-trop}
\max_{j\in [\ncols]}(A_{ij}+x_j)
 \leq \max_{j\in[\ncols]} (B_{ij}+x_j),\qquad i \in I 
\end{align}
Here, $I$ is a possibly infinite set, 
recall that $[\ncols]:=\{1,\ldots,\ncols\}$,
and $A_{ij},B_{ij}$ belong
to $\rmax$. When the previous system consists of finitely many
inequalities, i.e., when $I=[\nrows]$ for some integer $\nrows$, we obtain
a tropical {\em polyhedral cone}. Then, $A$ and $B$ will be thought
of as $\nrows\times \ncols$ matrices with entries in $\rmax$.
In the sequel, we shall denote by $\mnmatrices$ the set
of these matrices.

We look for a non-trivial element of the cone, i.e.,
for a solution $x=(x_j)\in \rmax^\ncols$ of the above system, not
identically $-\infty$.  From the algorithmic point of view, the polyhedral
case is of primary interest. However, some of our results
will turn out to hold as well in the case of infinite systems of inequalities,
and their relation with non-linear Perron-Frobenius theory will
be more apparent in this wider setting.

To study this satisfiability problem, we define the following zero-sum 
game, in which there are two players, ``Max'' and ``Min''
(the maximizer and the minimizer). The state space 
consists of the disjoint union of the set $I$ and the set
$[\ncols]$. The two players
alternate their moves. When the current position
is $i\in I$, Player Max must choose the next state $j\in[\ncols]$
in such a way that $B_{ij}$ is finite, and receives $B_{ij}$ from
Player Min. If Player Max does not have
any available action, i.e., if $B_{ij}=-\infty$ holds for all $j\in [\ncols]$,
then Player Max pays an infinite amount to player Min and the game
terminates.
Similarly, when the current state is $j\in[\ncols]$, 
Player Min must choose the next state $i\in I$ in such a way
that $A_{ij}$ is finite, and receives $A_{ij}$ from Player Max.
If $A_{ij}=-\infty$ holds for all $i\in I$, then,
player Min pays an infinite amount to player Max
and the game terminates. 

When $I=[\nrows]$, the game may be represented by a bipartite
(di)graph, with two classes of nodes, $[\nrows]$ and $[\ncols]$.
The players move alternatively a token on the graph,
following the arcs of the graph, which represent the possible moves.
The weight of an arc represent the associated payment, see Example~\ref{ex-first} below.

We shall often need the following assumptions, which require every player to have at least one available action in every state.
\begin{assump}\label{assump1}
For all $j\in [\ncols]$, there exists $i\in I$ such that
$A_{ij}\neq -\infty$.
\end{assump}
\begin{assump}\label{assump2}
For all $i\in I$, there exists $j\in [\ncols]$ such that
$B_{ij}\neq -\infty$.
\end{assump}
Systems of the form~\eqref{sys-trop} can always be transformed
to enforce these assumptions. First, the trivial inequalities
$x_j\geq x_j$, for $j\in [\ncols]$, can always be added
to the original system, which makes sure that Assumption~\ref{assump1} holds.
Then, if $B_{ij}=-\infty$ holds for some $i\in I$ and for all $j\in [\ncols]$, the right hand side of~\eqref{sys-trop} is identically
$-\infty$, and so every variable $x_j$ such that $A_{ij}\neq -\infty$ must be equal to $-\infty$. By eliminating the inequality corresponding to $i$,
we obtain a new system in the remaining variables which is equivalent
to the original one. We also eliminate all the inequalities in which
the left-hand side is identically $-\infty$ (which are trivially satisfied).
By performing this
elimination step a finite number of times, we eventually
arrive at an equivalent
system involving a subset of variables, and satisfying Assumption~\ref{assump2}.

Given an initial state $i$ and a horizon (number of turns) $N$, 
we define $v_i^N$ to be the value of the corresponding finite horizon
game for player Max. The existence of the value is immediate
when the horizon is finite (but the value may be infinite
if the set $I$ is infinite, or if 
Assumption~\ref{assump1} or Assumption~\ref{assump2} does not hold).

When both Assumptions~\ref{assump1} and~\ref{assump2} are fulfilled, we
 shall also consider the ``mean payoff'' game, in which
the payoff of an infinite trajectory is defined as the average
payment per turn received by player Max.  Formally, we define
this average payment as the limsup as the number $N$ of turns goes
to infinity of the payments received plus the opposite of the payments
made by Player Max up to turn $N$ divided by $N$. (When
Assumptions~\ref{assump1} or~\ref{assump2} do not hold, the
payments must be counted up to the termination time if the
latter occurs before time $N$.)
The value of such games was shown to exist by Ehrenfeucht and Mycielski~\cite{ehrenfeucht}, assuming that the state space is finite.
This can also be deduced from a theorem of
Kohlberg (Theorem~\ref{th-kohlberg} below),
which implies in addition that the value vector of this game
coincides with $\chi(f)$.

\begin{example}\label{ex-first}
The matrices
\[
A= \begin{pmatrix} 2 & -\infty\\
                       8  &   -\infty \\
                       -\infty & 0 
\end{pmatrix}
\qquad
B= 
\begin{pmatrix} 1 & -\infty\\
                       -3  &   -12 \\
                       -9 & 5 
\end{pmatrix}
\]
arising from the system
\begin{align}\label{systrop-ex}
\begin{array}{rl}
2+x_1&\leq 1+x_1 \\ 
8+x_1&\leq \max(-3+x_1,-12+x_2)\\
x_2&\leq \max(-9+x_1,5 +x_2)
\end{array}
\end{align}
yield the game
\[
\begin{picture}(0,0)%
\includegraphics{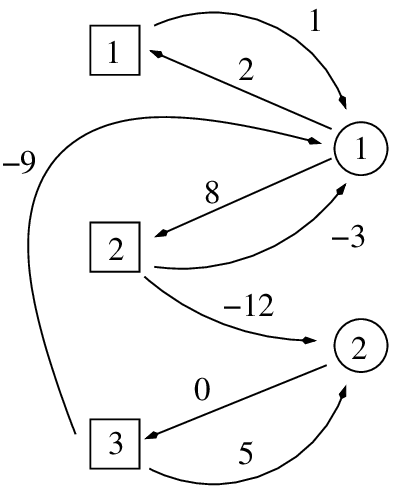}%
\end{picture}%
\setlength{\unitlength}{2072sp}%
\begingroup\makeatletter\ifx\SetFigFont\undefined%
\gdef\SetFigFont#1#2#3#4#5{%
  \reset@font\fontsize{#1}{#2pt}%
  \fontfamily{#3}\fontseries{#4}\fontshape{#5}%
  \selectfont}%
\fi\endgroup%
\begin{picture}(3557,4404)(751,-4945)
\end{picture}%

\]
in which the rows (states in which Max plays)
are denoted by squares and the columns (states in which Min plays)
are denoted by circles. For instance,
when in circle state $1$, Player Min may move
the token  to square state $1$, receiving a payment of $2$,
and then, Player Max has no choice but putting
the token on column node $1$, getting
back a payment of $1$ from Player Min. Hence, the mean
payoff per turn for Player Max will be $-2+1=-1$
if the initial state is the circle node $1$, and
if Player Min chooses this strategy. However, if Player Min 
moves (greedily) the token to square node $2$, receiving
8, Player Max may move the token to circle node $2$, paying
12 to Player Min, but then, Player Min can be forced
to follow the circuit between circle node $2$ and square
node $3$, which ensures a mean payoff of $5$ per
turn unit to Player Max. Using these observations, one can check that the 
value of this mean payoff game is $-1$ if the initial state is circle node $1$,
whereas the value is $5$ if the initial state is circle node $2$. 
We will explain below, in Example~\ref{ex-second}, how this
game can be solved analytically.
\end{example}
\subsection{Properties of order preserving and additively homogeneous maps}
We now recall some basic properties of the dynamic
programming operators arising from the previous games. 

We shall think of the collection of rewards $(B_{ij})_{i\in I,\;j\in[\ncols]}$
as a kernel, to which we associate the max-plus linear operator
$B:(\Rm)^\ncols\to (\Rm)^I$, 
\[
(Bx)_i:=\max_{j\in[\ncols]} (B_{ij}+x_j) ,\qquad \forall i\in I \enspace .
\]
When $I=[\nrows]$, $(B_{ij})_{i\in[\nrows],j\in[\ncols]}$ will be thought of as a matrix in $\mnmatrices$,
and $Bx$ is the product in the tropical sense of the matrix $B$ and the vector $x$.

When Assumption~\ref{assump2} holds, this operator sends $\R^\ncols$
to $\R^I$. We define the operator $A$ in the same way.
The {\em residuated operator} $A\res$ from 
$(\R\cup\{\pm\infty\})^I$ to $(\R\cup\{\pm\infty\})^\ncols$ is defined by
\begin{align}\label{e-def-res}
(A\res y)_j=\inf_{i\in I} (-A_{ij}+y_i) \enspace ,
\end{align}
with the convention $(+\infty)+(-\infty)=+\infty$. This operator
sends $(\R\cup\{-\infty\})^I$ to $(\R\cup\{-\infty\})^n$ whenever Assumption~\ref{assump1} holds, it sends $\R^I$ to $\R^n$ when in addition $I$ is finite.

The term residuated refers to the property
\begin{align}\label{eq-def-res}
Ax\leq y\iff x\leq A\res y \enspace ,
\end{align}
where $\leq$ is the partial order of $(\R\cup\{\pm\infty\})^I$
or $(\R\cup\{\pm\infty\})^n$.
Hence, System~\eqref{sys-trop}, which can be rewritten as $Ax\leq Bx$, is equivalent to $x\leq f(x)$ where $f:(\R\cup\{-\infty\})^n\to (\R\cup\{\pm\infty\})^n$ is defined by
\[ f(x):=A\res B x \enspace,
\]
denoting by concatenation the composition of operators.
The map $f$ sends $(\R\cup\{-\infty\})^n$ to itself  whenever 
Assumption~\ref{assump1} holds. It sends
$\R^n$ to $\R^n$ when in addition Assumption~\ref{assump2} holds
and $I$ is finite.

The map $f$ is
the dynamic programming operator of the previous game, meaning
that the vector $v^N:=(v^N_j)_{j\in[\ncols]}$ of values of the 
game in finite horizon can be computed recursively as follows
\[
v^N=f(v^{N-1}),\qquad v^0 =0 \enspace. 
\]
More generally, setting $v^0:=x$ for some $x\in \R^n$
determines the value function of a variant of the game, in which
after the last step, Player Min pays to Player Max a final amount
$x_j$ depending on the final state $j$.

We shall call {\em min-max functions} the self-maps 
of $(\Rm)^n$ that are of the form $A\res B$, when
$A,B\in \mnmatrices$.
This terminology goes back to Olsder~\cite{olsder91} and
Gunawardena~\cite{jmy}. However, 
unlike in the latter reference, we do not require
a min-max function to send $\R^n$ to $\R^n$.
This generality will be needed in Section~\ref{sec-linindep},
in which the games arising from the tropical linear independence
problem will turn out to have occasionally empty sets
of actions for Player Max.

Any min-max function $f$ from $(\R\cup\{-\infty\})^n$ to itself 
satisfies the following properties:
\begin{subequations}\label{mon-hom}
\begin{align}
f \text{ is \NEW{order-preserving}: }& x\leq y\Rightarrow f(x)\leq f(y)\quad
\forall x,y\in (\R\cup\{-\infty\})^n\enspace ,\label{mon-hom1}\\
f \text{ is \NEW{additively homogeneous}: }&
f(\lambda +x)=\lambda+f(x)\quad \forall \lambda\in \R\cup\{-\infty\},\; 
x\in (\R\cup\{-\infty\})^n\enspace ,\label{mon-hom2}\\
f \text{ is continuous.}&
\end{align}
\end{subequations}
Here, $\Rm$ is equipped with the usual topology,
defined for instance by the distance $(x,y)\mapsto |\exp(x)-\exp(y)|$,
and $(\Rm)^n$ is equipped with the product topology.

When an order-preserving and additively homogeneous map $f$ preserves $\R^n$,
it is easily seen to be sup-norm {\em nonexpansive}, meaning
that
\[
\|f(x)-f(y)\|\leq \|x-y\|,\qquad \forall x,y\in \R^n\enspace ,
\]
where $\|x\|=\max_{i\in[n]} |x_i|$. A min-max function that preserves
$\R^n$ is piecewise affine (we can cover
$\R^n$ by finitely many polyhedra in such a way that the restriction
of the function to each polyhedron is affine). Hence, the following
general result applies in particular to such min-max functions.
\begin{theorem}[Kohlberg~\cite{kohlberg}]\label{th-kohlberg}
A self-map of $\R^n$ that is nonexpansive in any norm
and piecewise affine
admits an invariant half-line, meaning that there exist
two vectors $v,\eta\in \R^n$ such that
\begin{align}\label{e-halfline}
f(v+t\eta)=v+(t+1)\eta 
\end{align}
for all scalars $t$ large enough. 
\end{theorem}
In the study of mean payoff games, an important issue is to 
determine the limit
\[
\chi(f):=\lim_{N\to\infty} f^N(0)/N=\lim_{N\to\infty} v^N/N \enspace,
\]
which gives the additive growth rate of the value of the finite horizon game
as a function of the horizon $N$.
Kohlberg's theorem implies that the limit $\chi(f)$ does exist.
\begin{corollary}\label{coro-kohl}
Assume that every player has at least one available action in
every state (Assumptions~\ref{assump1}-\ref{assump2}) and that
the state space is finite.
Then, 
\[
\chi(f)=\eta \enspace ,
\]
where $(v,\eta)$ is an arbitrary invariant
half-line of $f$. 
\end{corollary}
\begin{proof}
This result is well known in the operatorial approach of zero-sum games
(see in particular~\cite{neymansurv,rosenbergsorin}),
we include the (simple) proof for the convenience of the reader.
Since $f$ is nonexpansive in the sup-norm,
$\|f^N(x)-f^N(y)\|/N\leq \|x-y\|/N$ tends to $0$ as $N$ tends to infinity.
It follows that the existence and the value of the limit
\[
\lim_{N\to\infty} f^N(x)/N
\]
are independent of the choice of $x\in\R^n$. 
If $v,\eta$ is an invariant half-line,
choosing
$x=v+t_0\eta$, for some large enough $t_0$, 
we deduce that $f^N(x)=x+N\eta$, and so,
$\lim_{N\to\infty} f^N(x)/N=\eta$.
\end{proof}
\begin{remark}
If $f$ is an order-preserving and additively
homogeneous map preserving $\R^n$, then, it was
observed independently by Rubinov and Singer~\cite{singer00} 
and by Gunawardena and Sparrow (see~\cite{mxpnxp0}) that
\begin{align}
f(x)= \inf_{y\in\R^n} \left( f(y)+ \max_{j\in[\ncols]} (x_j-y_j) \right)\quad
\forall x\in \R^n\enspace .
\label{e-primal}
\end{align}
This shows that $f$ can be represented in the form $f=A\res B$
(the set $I$ is equal to $\R^n$). 
This applies in particular
to the dynamic programming operator of a stochastic zero-sum game
\begin{align}
f_j(x)= \inf_{a\in A_j}\sup_{b\in B_j}\big(r_j^{ab}+\sum_{k\in[n]}P_{jk}^{ab}x_k \bigr),\qquad j\in[n]\enspace,
\label{e-stoch}
\end{align}
where $[n]$ is the set of states at which the first player (Min) plays,
$A_j$ is the set of actions of this player,
$B_j$ is the set of actions of the second player (Max), 
$r_j^{ab}$ the payment made by Min to Max if the actions $a$ and $b$
are selected in state $j$, and then, $P_{jk}^{ab}$ is the probability
to move to state $k$. The map $f$ in~\eqref{e-stoch} is order preserving, additively homogeneous, and it preserves $\R^n$ if for instance the payments
$r_{j}^{ab}$ are bounded. Then, it can be represented
as in~\eqref{e-primal}. See~\cite{filar97} for more background
on stochastic games.
\end{remark}

\subsection{The Collatz-Wielandt property}

Some of the main results of this paper rely on a non-linear
version of the Collatz-Wielandt characterization
of the spectral radius which appears in Perron-Frobenius theory.

Given any self-map $f$ of $(\R\cup\{-\infty\})^n$ that is
order-preserving, additively homogeneous, and continuous,
we define the {\em Collatz-Wielandt number} of $f$ to be
\begin{align}
\cw(f) = \inf\{\mu\in \R\mid 
\exists w\in \R^n ,
f(w)\leq \mu +w \}\enspace.\label{eq-def-cw}
\end{align}
A vector $u\not\equiv-\infty$ is a (non-linear) eigenvector 
of $f$ for the eigenvalue $\lambda\in \Rm$ if 
\[
f(u)=\lambda +u \enspace .
\]
The (non-linear) {\em spectral radius} of $f$ is defined
as the supremum of its eigenvalues
\[
\rho(f)= \sup\{\lambda\in \R\cup\{-\infty\}\mid 
\exists u\in (\R\cup\{-\infty\})^n, \;u\not\equiv-\infty,\;
f(u)=\lambda +u \}\enspace ,
\]
and is itself an eigenvalue of $f$.
We shall also need the following ``symmetrical'' version
of the Collatz-Wielandt number
\[
\cw'(f):=
\sup\{\lambda\in \R\cup\{-\infty\}\mid 
\exists u\in (\R\cup\{-\infty\})^n, \;u\not\equiv-\infty,\;
f(u)\geq \lambda +u \}\enspace ,
\]
as well as the following quantity.
\begin{proposition}\label{exist-chibar}
If $f$ is an order-preserving additively homogeneous self-map of $(\Rm)^n$,
then, for all $x\in\R^n$, the following limit exists and is independent of
the choice of $x$: 
\begin{align}\label{defbarchi}
\bar\chi (f) :=\lim_{N\to\infty} \max_{j\in[\ncols]} f^N_j(x)/N\enspace .
\end{align}
\end{proposition}
\begin{proof}
This is a variant of a result established 
in~\cite{vinc94,gunawardenakeane,sgjg04} when $f$ preserves $\R^n$.
The proof relies on a simple subadditive argument, which
we adapt here for completeness.
Although $f$ may not preserve $\R^n$, it is still nonexpansive on  $\R^n$
in a generalized sense, meaning that for all $x,y\in\R^n$,
$-\|x-y\|+f(x)\leq f(y)\leq \|x-y\|+f(x)$.  
The same is true if we replace $f$ by its $N$th iterate $f^N$, and so,
$-\|x-y\|/N+ \max_{j\in[\ncols]} f^N_j(x) /N \leq 
\max_{j\in[\ncols]} f^N_j(y) /N 
\leq \|x-y\|/N+ \max_{j\in[\ncols]} f^N_j(x) /N$,
which shows that the existence and the value of the limit in~\eqref{defbarchi}
are independent of the choice of $x\in\R^n$. Hence, we take $x=0$
in what follows. Since $f$ is order-preserving and additively homogeneous, 
the sequence $t_N:= \max_{j\in[\ncols]} f^N_j(0)$ is easily
seen to be subadditive, meaning that $t_{N+M}\leq t_N+t_M$.
It follows that the limit 
$\bar\chi(f)=\lim_{N\to\infty} t_N/N$ exists.
\end{proof}
Of course, when $\chi(f)$ exists, we readily deduce from the definitions that
\[
\bar\chi(f)=\max_{j\in[\ncols]} \chi_j(f)  \enspace .
\]

We next derive the following lemma from a theorem 
of Nussbaum~\cite[Theorem~3.1]{MR818894}, dealing with non-linear
maps on finite dimensional cones, which
implies that $\cw(f)=\rho(f)$.
When $f$ preserves
$\R^n$, the fact that $\cw(f)=\bar\chi(f)$, together
with the final part of the statement of the lemma
 was proved in~\cite[Theorem~8]{sgjg04}.
We refer the reader to~\cite{Nuss-Mallet} for infinite
dimensional generalizations.

\begin{lemma}[Collatz-Wielandt property, compare with~\cite{MR818894} and~\cite{sgjg04}]\label{lemma-allequal}
Let $f$ denote a map from $(\R\cup\{-\infty\})^n$ to itself,
that is order-preserving, additively homogeneous, and continuous.
Then,
\begin{align}
\cw'(f)= \rho(f)=\cw(f)= \bar\chi (f) \enspace .
\label{allequal}
\end{align}
Moreover, there is at least one coordinate $j\in[\ncols]$ such that
$\chi_j(f):=\lim_{N\to\infty} f_j^N(x)/N$ exists and is equal
to $\bar\chi(f)$.
\end{lemma}
\begin{proof}%
Nussbaum~\cite[Theorem~3.1]{MR818894} showed that if $F$ is a continuous
self-map of a closed convex cone $C$ in $\R^n$, of nonempty interior,
which preserves
the order of the cone, and is positively homogeneous,
then, the {\em cone spectral radius} of $F$,
which is defined as 
\[
\max\{\lambda\in \R_+\mid 
\exists u\in C\setminus\{0\},\;
F(u)=\lambda u \}\enspace,
\]
where $\R_+$ denotes the set of nonnegative real numbers, coincides with
\begin{align}
\inf\{\mu\in \R_+\mid 
\exists w\in \operatorname{int}(C),
F(w)\leq \mu w \}\enspace.
\end{align}
We associate to a map $f$ the map $F(x)=\exp(f(\log x))$, where
$\log$ denotes the map from $\R_+^n$ to $(\R\cup\{-\infty\})^n$ which
does $\log$ entrywise (with $\log(0)=-\infty$),
 and $\exp$ denotes the inverse of this map. Applying
Nussbaum's theorem to the map $F$, we obtain 
$\rho(f)=\cw(f)$.

If $w\in \R^n$ and $\mu\in\R$ satisfy $f(w)\leq \mu+ w$, then, 
since $f$ is order preserving and additively homogeneous,
we get $f^N(w)\leq N\mu +w$, for all $N\geq 0$. It follows that 
$\bar\chi(f)=\lim_{N\to\infty} \max_{j\in[\ncols]} f^N_j(w)/N\leq \mu$.
Taking the infimum over all $\mu$ and $w$, we deduce that
$\bar{\chi}(f) \leq \cw(f)$. 

If $f(u)\geq \lambda +u$ for some $u\in (\R\cup\{-\infty\})^n$, not
identically $-\infty$, let $w$ be a vector obtained by replacing every
infinite entry of $u$ by an arbitrary finite entry. Since $u\leq w$, we
deduce that $f^N(w)\geq f^N(u)\geq N\lambda+u$. Since 
$\max_{j\in[\ncols]} u_j \neq -\infty$, we get 
$\bar \chi(f)=\lim\limits_{N\to\infty} \max_{j\in[\ncols]} f^N_j(w)/N \geq \lim\limits_{N\to\infty} (N\lambda +\max_{j\in[\ncols]} u_j)/N= \lambda$. 
It follows that $\lambda \leq \bar{\chi}(f)$.
Taking the supremum over all $\lambda$ and $u$, we deduce that
$\cw'(f)\leq \bar{\chi}(f)$.
We also observe that by definition, $\rho(f)\leq \cw'(f)$.
Hence, we obtain $\rho(f)\leq \cw'(f)\leq \bar{\chi}(f)\leq  \cw(f)=
\rho(f)$, from which we deduce the equality~\eqref{allequal}.

Finally, let us denote by $ \cw'_j(f)$ the supremum of the $\lambda$ 
such that there exists $u\in (\R\cup\{-\infty\})^n$ 
such that $u_j$ is finite and $f(u)\geq \lambda +u$.
By definition $\cw(f)=\max_{j\in[\ncols]} \cw_j(f)$, 
hence there exists $j\in[\ncols]$ such that $\cw(f)=\cw_j(f)$.
Let us fix such an index $j$.
For all $\lambda $, $u$ and $w$ as above and such that $u_j$ is finite, 
and for all $x\in\R^n$, we
deduce from  $\|x-w\|+f^N(x)\geq f^N(w)\geq f^N(u)\geq N\lambda+u$ that
$\liminf_{N\to\infty} f^N_j(x)/N \geq \lambda$. 
Taking the supremum over $\lambda$ and $u$, it
follows that $\liminf_{N\to\infty} f^N_j(x)/N\geq \cw'(f)=\bar\chi(f)$.
Since $\liminf_{N\to\infty} f^N_j(x)/N
\leq \limsup_{N\to\infty} f^N_j(x)/N\leq \bar\chi(f)$, we deduce that
the limit $\chi_j(f)$ exists and is equal to $\bar\chi(f)$.
\end{proof}
\begin{remark}\label{rk-iter}
For all integers $k$, it follows from the definition of $\bar\chi(f)$ that
$\bar\chi(f^k)=k\bar\chi(f)$. Hence, the same is true
for $\rho(f)$, $\cw'(f)$, and $\cw(f)$.
\end{remark}
\begin{remark}\label{rk-dual}
When $f$ preserves $\R^n$, all the previous notions can be dualized.
Indeed, it is known that an order preserving, additively homogeneous
self-map $f$ of $\R^n$ admits a unique continuous extension to $(\Rm)^n$,
which we will still denote by $f$, see~\cite{burbanks}.
Similarly, $f$ admits a unique continuous extension to $(\R\cup\{+\infty\})^n$,
which we also denote by $f$.  Then, the following dual quantities
\begin{align*}
&\sup\{\mu\in\R\mid \exists w\in\R^n,\qquad f(w)\geq \mu + w\}\enspace,\\
&\inf\{\lambda\in \R\cup\{+\infty\}\mid 
\exists u\in (\R\cup\{+\infty\})^n, \;u\not\equiv+\infty,\;
f(u)=\lambda +u \}\enspace,\\
&\inf\{\lambda\in \R\cup\{+\infty\}\mid 
\exists u\in (\R\cup\{+\infty\})^n, \;u\not\equiv+\infty,\;
f(u)\leq \lambda +u \}\enspace ,\\
&\underline{\chi}(f):=\lim_{N\to\infty}
\min_{j\in[n]}f_j^N(x)/N,\qquad x\in\R^n,
\end{align*}
are readily seen to coincide (apply Lemma~\ref{lemma-allequal} to the map
$x\mapsto -f(-x)$).
\end{remark}
\begin{remark}
When $f$ preserves $\R^n$, we get $\overline{\chi}(f)=\cw(f)\leq \max_{i\in[n]}(f_i(w)-w_i)<\infty$
for any vector $w\in\R^n$. By duality (Remark~\ref{rk-dual}),
$\underline{\chi}(f)>-\infty$. 
Since $\rho(f) = \overline{\chi}(f)\geq \underline{\chi}(f)$,
this implies in particular that $\rho(f)$ is finite.
\end{remark}

\subsection{From spectral theory to mean payoff games}
We now interpret the previous results in terms of games.
Whereas the ``nonexpansive maps'' approach of zero-sum
games is well known~\cite{neymansurv,rosenbergsorin}, the
significance in terms of games of the Collatz-Wielandt property
that we show in Proposition~\ref{prop-trop2game} does not
seem to have been noted previously (it shows that the value
is always well defined if one player is allowed to select
the initial state, without the usual compactness and regularity
assumptions).

We call {\em positional strategy} of Player Max
a map $\sigma:I\to [\ncols]$ such that $B_{i\sigma(i)}$
is finite for all $i\in I$ (so $\sigma$ is a rule,
telling to Player Max to move to state $\sigma(i)$
when the current state is $i$). Similarly, we
call {\em positional strategy} of Player Min
a map $\pi: [\ncols]\to I $ such that $A_{\pi(j)j}$
is finite for all $j\in [\ncols]$. We denote
by $\Sigma_{\max}$ and $\Sigma_{\min}$ the set of such
strategies, respectively. We associate to the positional
strategies $\sigma$ and $\pi$ the ``one-player'' dynamic programming maps
$g^\sigma$ and $h^\pi$,
\[
g^\sigma_j(x):=\inf_{i\in I} (-A_{ij}+ B_{i\sigma(i)}+x_{\sigma(i)})
\]
\[
h^\pi_j(x)=-A_{\pi(j)j}+\max_{k\in[\ncols]}(B_{\pi(j)k}+x_k) \enspace .
\]
Hence, $(g^\sigma)^N_j(0)$ represents
the minimal amount that Player Min will have to pay to Player
Max in $N$ steps, if the initial state is $j$, provided
that Player Max applies the positional strategy $\sigma$. 
A dual interpretation holds for $(h^\pi)^N_j(0)$.

By construction of $g^\sigma$ and $h^\pi$, we have
\begin{align}
f(x)=\sup_{\sigma\in \Sigma_{\max}}g^\sigma(x) =\inf_{\pi\in\Sigma_{\min}} h^\pi(x) \enspace,\qquad\forall x\in\R^n
\enspace,\label{e-selectweak}
\end{align}
moreover, there is at least one $\sigma$, depending on $x$, attaining
the supremum (take $\sigma(i)=k$ where $k$ attains the maximum in
$\max_{k}(B_{ij}+x_k)$. Similarly, the infimum is attained by at least
one $\pi$, also depending on $x$, when the set $I$ is finite.

The following proposition, which is a consequence
of the Collatz-Wielandt property, shows that $\bar\chi(f)$ can be interpreted
as the {\em value} of a variant of the mean payoff game in which 
the choice of the initial state belongs to Player Max.
The lack of symmetry between both players is 
due to the fact that the set of states in which Max plays, i.e., the set $I$,
can be infinite.
\begin{proposition}\label{prop-trop2game}
Make Assumptions~\ref{assump1}, \ref{assump2}. Then,
Player Max can choose an initial node $j\in[\ncols]$,
together with a positional strategy,
so that he wins a mean payoff of at least $\bar\chi(f)$,
whatever strategy Player Min chooses. Moreover, for all $\lambda>\bar\chi(f)$,
Player Min can choose a positional
strategy so that she looses a mean payoff
no greater than $\lambda$
for all initial nodes $j\in[\ncols]$, whatever
strategy Player Max chooses.
\end{proposition}
\begin{proof}
Let us take arbitrary scalars $\lambda,\mu$ such that $\lambda>\mu>\bar\chi(f)$.
Since $\bar\chi(f)=\cw(f)$, we can find a vector $u\in\R^n$ such that
$f(u)\leq \mu + u $. Since $\lambda>\mu$
and 
\begin{align}
f_j(u)= \inf_{i\in I} \biggl( -A_{ij}+\max_{k\in[\ncols]} (B_{ik}+u_k) \biggr) 
\label{e-ext}\end{align}
we can choose,
for all $j\in [\ncols]$, an index $i=\pi(j)$ such
that $-A_{ij}+\max_{k\in[\ncols]} (B_{ik}+u_k) $ is close enough
to the above infimum,
and then,
\[
h^\pi(u)\leq \lambda + u \enspace .
\]
Since $h^\pi$ is order preserving and additively homogeneous,
an inequality of the form $h^\pi(u)\leq \lambda+u$
implies that $(h^\pi)^N(u)\leq N\lambda+u$. Moreover, since 
$0\leq \|u\|+u$, we deduce that $(h^\pi)^N(0)\leq \|u\|+(h^\pi)^N(u)$,
and so, 
\[
(h^\pi)^N(0)\leq 2\|u\|+ N\lambda  \enspace .
\]
Thus, if Player Min applies the positional strategy $\pi$, and if the
initial state is $j\in[\ncols]$, she is guaranteed to loose
no more than
\[
\limsup_N (h^\pi)^N_j(0)/N \leq \lambda \enspace,
\]
whatever strategy Player Max selects. 

Take now $u$ to be an eigenvector of $f$ for the eigenvalue $\rho(f)$,
so that $f(u)=\rho(f)+u$. Since the maxima arising in the expressions~\eqref{e-ext}
are taken over finite, non-empty, sets, Player Max can choose a positional
strategy $\sigma$ in such a way that
\[
f(u)=g^\sigma(u) \enspace .
\]
Then,
\[ f^N(u) = (g^\sigma)^N(u)=N\rho(f)+ u \enspace .
\]
Moreover, we can write $0\geq \alpha +u$, where $\alpha:=
- \max_{k\in[\ncols]} u_k \in\R$,  and so
$(g^\sigma)^N(0)\geq \alpha +(g^\sigma)^N(u)=\alpha+N\rho(f)+u$.
It follows that
\[
\liminf_N (g^\sigma)^N_j(0) /N \geq \rho(f) \qquad \text{ if}\; u_j\neq\-\infty \enspace .
\]
Hence, the positional strategy $\sigma$ guarantees to Player Max
the win of a mean payoff of at least $\rho(f)$ if the initial state $i$
is such that $u_i\neq-\infty$.
\end{proof}

When the set of actions is finite on both sides,
the previous analysis can be simplified. Actually, we have
the following strong duality
result.
\begin{theorem}[{Coro.\ of~\cite{kohlberg}, or~\cite{gg98}}]
\label{th-dual}
Make Assumptions~\ref{assump1}, \ref{assump2}, and assume that
$I$ is finite (so, both players have finite actions sets). We have
\begin{align}
\chi(f)=\max_{\sigma}\chi(g^\sigma)= \min_{\pi}\chi(h^\pi) \enspace,
\label{e-dualth}
\end{align}
where the supremum and the infimum are taken over the 
positional strategies $\sigma\in\Sigma_{\max}$ and $\pi\in\Sigma_{\min}$ of players Max and Min,
respectively. 
In particular, Player Max can choose a positional strategy, so that he wins a mean payoff of at least $\chi_j(f)$ if the initial node is $j$,
whatever strategy Player Min chooses. Similarly, Player Min can choose a positional strategy, so that she looses a mean payoff no greater than $\chi_j(f)$
if the initial node is $j$, whatever strategy Player Max chooses.
\end{theorem}
The first (and main) statement of this theorem appeared in~\cite{gg98}, as a consequence
of the termination of a policy iteration algorithm for mean payoff games,
building on the ideas of~\cite{gg}.
A simpler argument which applies more generally to stochastic games with perfect
information and finite state and action spaces appeared in~\cite{gg98a}.
Actually, the result can be quickly derived from
the existence of invariant half lines established by Kohlberg~\cite{kohlberg}.
We include the latter derivation here,
since it shows how optimal strategies are effectively obtained
from invariant half-lines. 
Note that Theorem~\ref{th-dual} is related to, but different from,
a strong duality theorem of Liggett and Lippman~\cite{liggettlippman},
see Remark~\ref{rk-related} below. 
\begin{proof}
The map $f\mapsto \chi(f)$ is order preserving. For all $\sigma$,
we have $g^\sigma\leq f$, and 
so
\[
\chi(g^\sigma) \leq \chi(f) ,\qquad \forall \sigma \enspace  .
\]
Let $v,\eta$ denote an invariant half-line of $f$,
so that $f(v+t\eta)=v+(t+1)\eta$ for $t$ large enough.
We shall use the fact that the set of
scalar affine functions $t\mapsto \varphi(t):=a+tb$, with $a,b\in\R$,
is totally ordered for the pointwise ordering $\leq$ in a neighborhood of $\infty$, which is such that $\varphi_1\leq \varphi_2$ if the inequality
$\phi_1(t)\leq \varphi_2(t)$
holds for all $t$ large enough. Actually, this order is nothing
but the lexicographic order on the coefficients, i.e., 
$\varphi_1(t)=a_1+tb_1\leq \varphi_2(t)=a_2+tb_2$ if $b_2>b_1$ or
$b_2=b_1$ and $a_2\geq a_1$. In particular, the supremum or the
infimum of a finite family of affine functions of $t$ coincides with one of these affine functions of $t$, for $t$ large enough. It follows that, for every $i\in I$, 
we can choose $\sigma(i)$ such that
\[
\max_{k\in[n]}(B_{ik}+v_k+t\eta_k)= B_{i\sigma(i)}+v_{\sigma(i)}+t\eta_{\sigma(i)}
\]
holds for $t$ large enough. 
Then, $v+(t+1)\eta=f(v+t\eta)=g^\sigma(v+t\eta)$
for $t$ large enough, showing that $v,\eta$ is also an invariant half-line
of $g^\sigma$. It follows from Corollary~\ref{coro-kohl}
that $\chi(f)=\chi(g^\sigma)$, showing that $\sigma$
attains the maximum in~\eqref{e-dualth}. The argument for $\pi$ is similar.
\end{proof}
\begin{remark}\label{rk-related}
Theorem~\ref{th-dual} should be compared with the strong
duality theorem of Liggett and Lippman~\cite{liggettlippman},
which concerns the value of the mean payoff game,
in which the payment of an infinite run is defined as
the limsup of the payment per turn as the number of turns
tends to infinity, instead of $\chi(f)$, defined to be
the limit of the value per turn of the finite horizon game.  
However, the existence of invariant half-lines (Kohlberg's theorem) implies
that the limit and value operations commute, i.e., that
$\chi(f)$ coincides with the value of the mean payoff game.
In fact, the strategies $\sigma$ and $\pi$
constructed in the proof of Theorem~\ref{th-dual} are
easily seen to be optimal for the latter game. 
Hence, knowing Kohlberg's theorem, one can deduce 
Theorem~\ref{th-dual}, or rather, its generalization
to the case of stochastic games with perfect information and finite state and action spaces, from the result of Liggett and Lippman~\cite{liggettlippman} and vice versa. 
\end{remark}
\begin{remark}
Theorem~\ref{th-dual} provides a good characterization in the sense
of Edmonds of the limit value per turn, $\chi(f)$.
The strategy $\sigma$
attaining the maximum in~\eqref{e-dualth} provides a concise certificate
allowing one to make sure that $\chi(f)$ is greater than or equal
to a given vector. Indeed, $g^\sigma$ is the dynamic programming
operator of a {\em one player} problem, and so, $\chi(g^\sigma)$
can be computed in polynomial time, by reduction to the
maximal circuit mean problem~\cite{karp78}. Similarly, 
the strategy $\pi$ attaining the minimum in~\eqref{e-dualth} provides a concise certificate allowing one to make sure that $\chi(f)$ is smaller than or equal
to a given vector. This is illustrated in Example~\ref{ex-geom} below.
\end{remark}
\section{The correspondence between tropical convexity and mean payoff games}
\subsection{The reductions}
We now come back to our original system of inequalities~\eqref{sys-trop}, 
written as $Ax\leq Bx$ for brevity.
We associate to this system
the mean payoff game with dynamic programming operator $f=A\res B$.

Proposition~\ref{prop-trop2game} makes it legitimate to say that Player
Max has a {\em winning position} whenever $\bar\chi(f)\geq 0$ (i.e.,
Player Max can choose the initial state in such a way that the mean payoff game has a nonnegative value). More generally, we shall say that the initial
state $i$ is {\em winning} whenever $\chi_i(f)$ does exist and is nonnegative.

Our first result, which we deduce from the Collatz-Wielandt property (Lemma~\ref{lemma-allequal}), does not require the number of inequalities to be finite.

\begin{theorem}\label{theo-reduce}
Under Assumption~\ref{assump1}, the system of linear tropical inequalities
$Ax\leq Bx$ has a solution $x\in \rmax^n$ non-identically $-\infty$ if and only
if Player Max has a winning state in the mean payoff game with
dynamic programming operator $f(x)=A\res B x$.
\end{theorem}
\begin{proof}
Due to the residuation property~\eqref{eq-def-res}, 
we have that $Ax\leq Bx$ if and only if $x\leq f(x)$.

Hence, if the system $Ax\leq Bx$ has a solution $x\in \rmax^n$ not identically
$-\infty$, then, $\cw'(f)\geq 0$, and so, by Lemma~\ref{lemma-allequal},
$\bar\chi(f)\geq 0$, 
showing that Player Max has a winning state.
Conversely, if Player Max has a winning state, $\rho(f)=\bar\chi(f)\geq 0$,
and so, there exists a vector $u\in (\Rm)^{\ncols}$, not identically $-\infty$,
such that $f(u)=\rho(f)+u\geq u$. Then, $Au\leq Bu$.
\end{proof}

Actually, by using Kohlberg's invariant half-lines,
instead of the Collatz-Wielandt type property of Lemma~\ref{lemma-allequal},
we arrive at the following more precise result when the number
of inequalities is finite.

\begin{theorem}\label{theo-reduce2}
Let Assumptions~\ref{assump1} and \ref{assump2} be satisfied,
and suppose that the system $Ax\leq Bx$
consists of finitely many inequalities ($I=[\nrows]$).
Consider the polyhedral cone
$P:=\set{x\in \rmax^n}{Ax\leq Bx}$, and 
define the support $S$ of $P$ to be the
union of the supports of the elements of $P$:
\[ S:=\set{j\in [\ncols]}{\exists u\in P, u_j\neq-\infty}\enspace .
\]
Then $S$ is %
a support 
of an element of $P$, that is there
exists $u\in P$ such that $S=\set{j\in [\ncols]}{u_j\neq-\infty}$.
Moreover, $S$ coincides with the set of initial states with a nonnegative value
for the associated mean payoff game, that is:
\begin{equation}\label{S=T}
 S=\set{j\in [\ncols]}{\chi_j(f)\geq 0}\enspace ,\end{equation}
where $f:(\Rm)^n\to(\Rm)^n$ is such that $f(x)=A\res B x$.
\end{theorem}
\begin{proof}
The first assertion follows from the max-plus convexity of $P$, which is
itself due to the max-plus linearity of the maps $A$ and
$B$. Indeed, for all $j\in S$, let $u^{(j)}$ be an element of $P$ such that 
$u^{(j)}_j\neq -\infty$. Then $v:=\sup_{j\in S} u^{(j)}$ is an element of $P$
since $Av=\sup_{j\in S} Au^{(j)}\leq \sup_{j\in S} Bu^{(j)}=Bv$, and
$v_j\neq -\infty $ for all $j\in S$. Hence, 
 $S\subset \set{j\in [\ncols]}{v_j\neq-\infty}\subset S$, which shows that
$S$ is the maximal support of an element of $P$.

In order to prove the second assertion, 
let us denote by $T$ the right hand side of~\eqref{S=T}.
Let $u\in P$. Due to the residuation property~\eqref{eq-def-res}, 
we deduce that $u\leq f(u)$. It follows
that $u\leq f^N(u)$, for all $N$.
Hence, as in the proof of Lemma~\ref{lemma-allequal},
considering a vector $w$ obtained by replacing every
infinite entry of $u$ by an arbitrary finite entry, we obtain that 
$\chi_i(f)=\lim\limits_{N\to\infty} f^N_i(w)/N \geq \lim\limits_{N\to\infty} f^N_i(u)/N \geq \lim\limits_{N\to\infty} u_i/N= 0$
as soon as $u_i\neq -\infty$.
It follows that $S\subset T$.

Conversely,  let $v,\eta$ denote an invariant half-line of $f$
(given by Theorem~\ref{th-kohlberg}), so
that $f(v+t\eta)=v+(t+1)\eta$ for $t$ large enough, and $\chi(f)=\eta$
(by Corollary~\ref{coro-kohl}).
Then, $A(v+(t+1)\eta)\leq B(v+t\eta)$, i.e.
\begin{align}\label{e-cont}
\max_{j\in[\ncols]}(A_{ij}+v_j+(t+1)\eta_j ) \leq 
\max_{j\in[\ncols]} ( B_{ij}+v_j+t\eta_j) \enspace ,\qquad i\in[\nrows] 
\enspace .
\end{align}
Let $u\in\rmax^n$ be such that $u_j=v_j$ 
for all  $j\in T$, and $u_j=-\infty$ for all  $j\notin T$.
We shall show that, for $t$ large enough,
$A(u+t\eta)\leq B(u+t\eta)$, or equivalently 
\begin{align}
 \max_{j\in T}(A_{ij}+v_j+t \eta_j )\leq 
\max_{j\in T} (B_{ij}+v_j+t\eta_j) ,\qquad i\in[\nrows]\enspace .
\label{e-nice}
\end{align}

Indeed, since $\eta_j\geq 0$ when $u_j\not=-\infty$, we deduce that
$ A(u+t \eta )\leq A(u+(t+1) \eta )\leq A(v+(t+1) \eta )$
for all $t\geq 0$.
Moreover, by definition of $T$, we get that:
\begin{equation}\label{ineqB}
\max_{j\in T^c} (B_{ij}+v_j+t\eta_j )\leq M+t \mu
\quad \forall t\geq 0 \text{ and } i\in[\nrows], \end{equation}
where $T^c$ denotes the complement of $T$ in $[\ncols]$, $\mu=\max_{j\in T^c}\eta_j<0$ 
and $M$ is a real constant.
Hence, using~\eqref{e-cont}, we obtain, for all $t$ large enough:
\begin{equation}\label{eq-inter}
[A(u+t\eta)]_i\leq \max([B(u+t\eta)]_i, M+t \mu)\enspace.
\end{equation}

Again, using that $\eta_j\geq 0$ when $u_j\not=-\infty$, we deduce that
$Au\leq A(u+t \eta )\leq Au +t\bar\chi(f)$ for all $t\geq 0$.
When $[Au]_i=-\infty$, this implies that
 $[A(u+t \eta )]_i=-\infty$ for all $t\geq 0$,
so that $[A(u+t \eta)]_i\leq [B(u+t\eta)]_i$.
Otherwise, $[Au]_i\neq -\infty$, then 
$M+t\mu< [Au]_i\leq [A(u+t  \eta )]_i$ for $t$ large enough,
and by~\eqref{eq-inter}, 
we obtain that  $[A(u+ t\eta)]_i\leq [B(u+t\eta)]_i $.

This shows that $A(u+t\eta)\leq B(u+t\eta)$ holds for $t$ large enough.
Fixing such a $t$, we get that $u+t\eta\in P$.
Since $v\in\R^n$, the support of $u+t\eta$ is equal to $T$.
We have proved that $T\subset S$, and so, $T=S$.
\end{proof}
\begin{remark}
Theorem~\ref{theo-reduce2} shows in particular that $S$ is the maximal support of an element of $P$.
\end{remark}
The case of a full support in Theorem~\ref{theo-reduce2} leads to the
following result, 
which was already pointed out by Dhingra and Gaubert in~\cite{dhingra}.

\begin{corollary}[{\cite[\S IV,~C]{dhingra}}]\label{theo-reduce3}
Make Assumptions~\ref{assump1} and \ref{assump2},
and suppose that the system $Ax\leq Bx$ consists
of finitely many inequalities.
Then, this system has a solution $x\in \R^n$ if and only if
all the initial states of the associated game have a nonnegative value,
i.e.,
\[
\chi(f)\geq 0 \enspace .
\]
\end{corollary}

Rather than a tropical polyhedral cone, we now consider a {\em tropical polyhedron} $P$, which is defined by systems
of affine tropical inequalities of the form           
\begin{align}\label{sys-trop2}
\max(\max_{j\in[\ncols]}(A_{ij}+x_j),c_i) \leq \max(\max_{j\in[\ncols]} (B_{ij}+x_j),d_i),\qquad i\in[\nrows]
\end{align}
where the matrices $A,B$ are as above and $c_i,d_i\in \rmax$.

As in the case of classical convexity, polyhedra
can be represented by polyhedral cones, the latter being the projective
analogues of the former affine objects.  So, 
we construct new matrices $\hat{A}$ and $\hat{B}$ by completing the matrices $A$ and $B$ by an $(\ncols+1)$th column, in such a way that $\hat A_{i,\ncols+1}=c_i$ and $\hat B_{i,\ncols+1}=d_i$, for all $i\in[\nrows]$. 

We now define the map $\hat f(y):=\hat{A}\res \hat{B} y$ for all $y\in (\Rm)^{\ncols+1}$.

\begin{theorem}\label{coro-new}
The tropical polyhedron $P$ defined by~\eqref{sys-trop2} is nonempty
if and only if the value of the mean payoff game with dynamic programming
operator $\hat f$, starting from the initial state $n+1$, is nonnegative,
i.e., $\chi_{n+1}(\hat{f})\geq 0$. \qed
\end{theorem}
\begin{proof}
For any $x\in \rmax^\ncols$, we define the vector $\hat x$ by completing the vector $x$ by an $(\ncols+1)$th coordinate equal to $0$. Then,
\[ 
x\in P \iff \hat{x}\in C:=\{y\in \rmax^{\ncols+1}\mid \hat{A}y\leq \hat{B}y\} \enspace .
\]
Moreover, if $y\in C$ is such that  $y_{n+1}\neq -\infty$, then the vector $x\in\rmax^n$ such that $x_i=y_i-y_{n+1}$ for $i\in[n]$ belongs to $P$.
Hence, the polyhedron $P$ is nonempty if and only if there exists $y$ in the
polyhedral cone $C$ such that $y_{n+1}\neq -\infty$.

We may
assume, perhaps after some transformations,
that the matrices $\hat{A},\hat{B}$ satisfy Assumptions~\ref{assump1}
and~\ref{assump2}. Then, applying Theorem~\ref{theo-reduce2} to $\hat{A}$, $\hat{B}$ and $C$, we readily obtain the %
assertion of the theorem. 
\end{proof}

Note that this theorem
shows that the emptyness problem for (affine) tropical polyhedra reduces to checking whether
a mean payoff game has a specific winning state.

The next theorem yields the converse reduction.
\begin{theorem}\label{th-converse}
Let $f=A\res B$, with $A,B\in \mnmatrices$, denote the dynamic programming %
operator of a mean payoff game (thus Assumptions~\ref{assump1} and \ref{assump2} are satisfied). Then, for every $r\in [n]$ and $\lambda\in \mathbb{R}$, the inequality 
$\chi_r(f)\geq \lambda $ holds if and only if the following tropical polyhedron
is non-empty:
\[
P_r:=\{y\in \rmax^{J}\mid \lambda + \max(\max_{j\in J} (A_{ij}+y_j), A_{ir})
\leq \max(\max_{j\in J} (B_{ij}+y_j), B_{ir}),\;\forall i\in [m]\} \enspace ,
\]
where $J:=[n]\setminus \{r\}$.
\end{theorem}
\begin{proof}
Note first that $\chi(f)=\lambda +\chi(g)$, where $g$ is the dynamic map
of the mean payoff game obtained by adding the constant $\lambda$ to every
entry of $A$. Thus, it suffices to consider the case in which $\lambda=0$.
Then, Theorem~\ref{theo-reduce2} shows that $\chi_r(g)\geq 0$ if and only
if the polyhedron 
\[
P:=\{x\in \rmax^n\mid \max_{j\in [n]}( A_{ij}+x_j)\leq \max_{j\in [n]}
(B_{ij}+x_j)\}
\]
admits a solution $x$ such that $x_r\neq-\infty$. Setting $y_j=x_j-x_r$
for $j\in J=[n]\setminus \{r\}$, we get that $y\in P_r$, and vice versa.
\end{proof}
\begin{corollary}\label{coro-main}
Each of the following problems: 
\begin{enumerate}
\item Is an (affine) tropical polyhedron empty?
\item Is a prescribed initial state in a mean payoff game winning?
\end{enumerate}
can be transformed in linear time to the other one.
\end{corollary}
\begin{remark}
We could use Theorem~\ref{th-converse} together with a dichotomy
argument to compute the value of a mean payoff game using an oracle
solving the emptyness problem for tropical polyhedra. 
\end{remark}
In some circumstances, the matrices $A,B$ may have integer
coefficients, and we may be only interested in the elements
of a tropical polyhedron with integer (or $-\infty$) coordinates.
The following result shows that in this case, the feasibility problems over
the integers and over the reals are equivalent.
\begin{proposition}\label{prop-integer}
Let $A,B$ be two $m\times n$ matrices with entries in $\mathbb{Z}\cup\{-\infty\}$. Then, if the polyhedron 
\[
P:=\{x\in \rmax^n\mid \max_{j\in[n]}(A_{ij}+x_j)
\leq \max_{j\in[n]}(B_{ij}+x_j),\;\forall i\in[m]\}
\]
contains a vector $y$ with entries in $\mathbb{R}\cup\{-\infty\}$,
it also contains a vector $z$ with entries  in $\mathbb{Z}\cup\{-\infty\}$
and precisely the same set of indices (positions) of
finite coordinates.
\end{proposition}
\begin{proof}
Without loss of generality, we assume that all the coordinates of $y$
are finite (otherwise, it suffices to consider the pre-image of the
polyhedron $P$ by the injective map $x\in \rmax^J\mapsto \hat{x}\in \rmax^n$ 
where $J$ is set of 
indices of
finite entries of $y$ and $[\hat{x}]$ is obtained 
by completing $x$ by $-\infty$ entries: $[\hat{x}]_j=x_j$ for
$j\in J$ and $[\hat{x}]_j=-\infty$ for $j\not\in J$).
Then, we must
show that $P$ contains a vector $z\in \mathbb{Z}^n$. The condition
that $y\in P$ may be written as $y\leq f(y)$ where $f:=A\res B$.
Moreover, since $y\in\R^n$ is such that $y\leq f(y)$, 
it suffices to add or eliminate inequalities in the definition of 
$P$ in order to obtain matrices $A$ and $B$ satisfying 
Assumptions~\ref{assump1} and~\ref{assump2}. Hence, the 
map $f$ can be assumed to preserve $\R^n$. Then, it is nonexpansive
in the sup-norm, and so does the map $g(x):=\min(f(x),x)$. Since $y\leq f(y)$,
we have $g(y)=y$. Consider now the orbit $z^k=g^k(0)$, recalling that $g^k$
denotes the $k$th iterate of $g$. Since $g$ is nonexpansive, we have
$\|z^k-y\|= \|g^k(0)-g^k(y)\|\leq \|y\|$, which shows that $z^k$
is bounded as $k$ tends to infinity. By definition of $g$, we have
$z^0\geq z^1\geq \cdots$. Hence the sequence $z^k$ is converging to some 
element of $\R^n$. Since the coefficients $A_{ij}$ and $B_{ij}$
are integers (or $-\infty$), the map $g$ preserves $\mathbb{Z}^n$, and so,
the sequence $z^k$ must be ultimately stationary, meaning that $z^{k+1}=g(z^k)=z^k$ for some $k$. It follows that $z^k\leq f(z^k)$, and so, $P$
contains the vector $z^k$ which has (finite) integer coordinates.
\end{proof}
Since every affine polyhedron in dimension $n$ can be represented
by a polyhedral cone in dimension $n+1$ by the trick which we
used in the proof of Theorem~\ref{coro-new}, the following is obtained
as an immediate corollary.
\begin{corollary}\label{cor-affine}
Let $A,B$ be two $m\times n$ matrices with entries in $\mathbb{Z}\cup\{-\infty\}$, and let $c,d$ be two vectors of dimension $n$ with entries
in $\mathbb{Z}\cup\{-\infty\}$. Then, if the polyhedron 
\[
P:=\{x\in \rmax^n\mid \max(\max_{j\in[n]}(A_{ij}+x_j),c_i)
\leq \max(\max_{j\in[n]}(B_{ij}+x_j),d_i),\;\forall i\in[m]\}
\]
contains a vector $y$ with entries in $\mathbb{R}\cup\{-\infty\}$,
it also contains a vector $z$ with entries  in $\mathbb{Z}\cup\{-\infty\}$
and precisely the same set of 
indices of
finite coordinates.\qed
\end{corollary}

\begin{example} \label{ex-second}
Let us illustrate Theorems~\ref{theo-reduce} and~\ref{theo-reduce2}
on the base of Example~\ref{ex-first}. An invariant
half-line of $f$ is given by
\[
v= (0,\; 0 )^T, 
\qquad
\text{and}\qquad\eta=(-1,\; 5)^T
\]
The vector 
\[
v'= (-\infty,\; 0 )^T
\]
is a solution of~\eqref{sys-trop}. The fact that $\eta_1=-1<0$ implies
that there is no finite solution. (Recall that invariant half-lines can be
obtained from several mean payoff game algorithms, including~\cite{dhingra}.)
\end{example}
\begin{example}\label{ex-geom}
We now give a geometrical illustration of Theorem~\ref{theo-reduce2}.
Let $a$ denote a real parameter, and consider
the three inequalities
\begin{align*}
x_1&\leq a+\max(x_2-2,x_3-1)\quad &(H_1)\\
-2+x_2&\leq a+\max(x_1,x_3-1)\quad &(H_2)\\
\max(x_2-2,x_3-a)&\leq x_1+2\quad &(H_3)
\end{align*}
the tropical half-spaces $(H_i)$ and their intersection being shown in Figure~\ref{fig-inter}, for $a\in\{1,-1/2,-3/2\}$. 
Here, each finite vector $x\in\Rm^3$ is represented by its intersection
with the hyperplane $x_1+x_2+x_3=0$, thus, vectors with $-\infty$ coordinates
correspond to points at infinity in this picture.

The associated game is shown in Figure~\ref{fig-gameinter}.
When $a=1$, an invariant half-line of $f$ is
\[
v=(-1.5,\;0.5,\;0)^T,\qquad \eta=(1.5,\;1.5,\;1.5) \enspace .
\]
Choosing a positional strategy of Max is equivalent to selecting only one output arc in each
square state. Such a strategy $\sigma$ is shown in Figure~\ref{fig-gameinter} (middle). 
The value of the one player game in which $\sigma$ is fixed can
be easily seen to be $(2a+1)/2$. This corresponds to the payment-per-turn
ratio of the circuit shown in bold (recall that the weights of the arcs
from a circle to a square nodes must be counted negatively, whereas
the weights of arcs from square to circle nodes count positively;
moreover, each arc counts for a half-turn). There is in fact another circuit,
from circle node 3 to square node 2, and back, which has a payment-per-turn
ratio of $a+1$; since $a+1> (2a+1)/2$, the latter never arises in an optimal
reponse of Player Min to the strategy $\sigma$.
By Theorem~\ref{th-dual},
we have $\chi_i(f)\geq \chi_i(g^\sigma)=(2a+1)/2$ for all $i$,
which by Theorem~\ref{theo-reduce2}, implies that the intersection
of the three half-spaces $H_1\cap H_2\cap H_3$ is not reduced
to the identically $-\infty$ vector for all $a\geq -1/2$, which
can be checked geometrically in Figure~\ref{fig-inter}. 

Similarly, choosing a positional strategy of Min is equivalent to selecting only one output arc in each
circle state. Such a strategy $\pi$ is shown in Figure~\ref{fig-gameinter} (right).  Whenever Min chooses this strategy, Max has no better
choice than going to the same circuit in bold, showing
that $\chi_i(f)\leq \chi_i(h^\pi)=(2a+1)/2$ for all $i$. Thus,
$\chi_i(f)=(2a+1)/2$ for all $i\in[3]$. In particular,  $\chi(f)_i=\chi_i(h^\pi)<0$
for $a<-1/2$, and so, the strategy $\pi$ certifies that the intersection
$H_1\cap H_2\cap H_3$ is reduced to the identically $-\infty$ vector
in this case.

\begin{figure}[htbp]
\[
\input{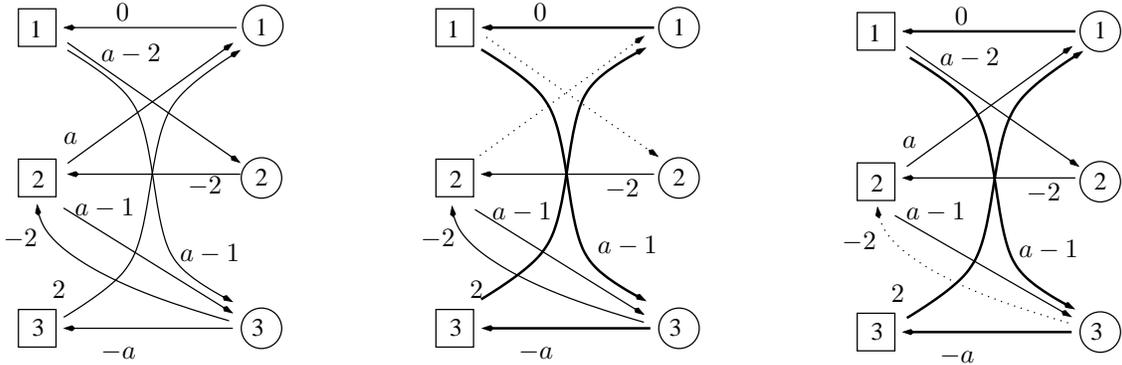}
\]
\caption{Game associated with the family of tropical half-spaces $H_i$, $1\leq i\leq 3$ 
in Example~\ref{ex-geom} (left). 
Positional strategies of Max (middle) and Min (right), both
with a mean payoff of $(2a+1)/2$, corresponding to the
circuit shown in bold (the strategies avoid the dotted arcs).}
\label{fig-gameinter}
\end{figure}
\begin{figure}[htbp]
\input{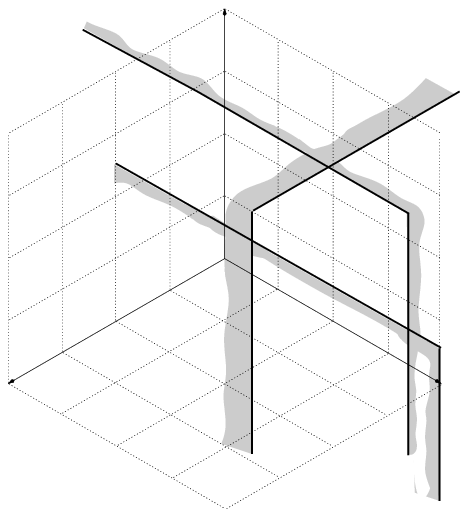}\hskip-0.1em
\input{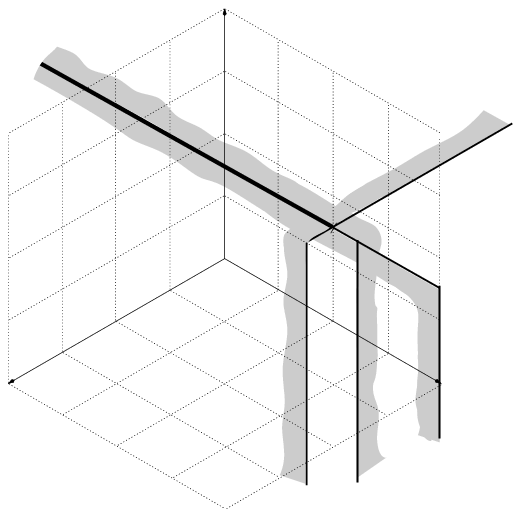}%
\input{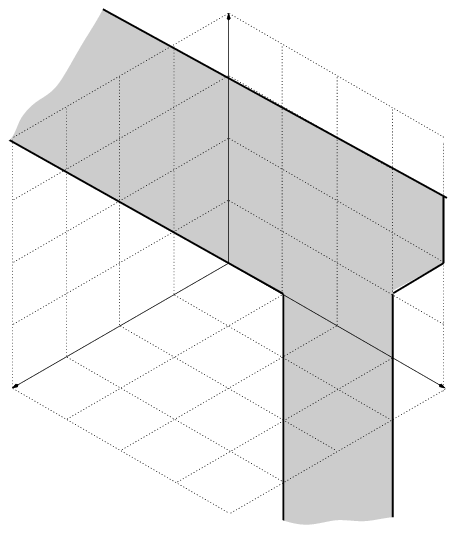}%
\caption{Example~\ref{ex-geom} (cont.); $a=-3/2$, a relative neighborhood of the boundary of each half-space is shown in gray, the intersection does not contain any finite vector (left);
$a=-1/2$, the intersection is a tropical convex cone with two generators, represented by a half-line in bold (middle); $a=1$, the intersection 
of half-spaces $H_1\cap H_2\cap H_3$ is the region in gray (right).}\label{fig-inter}
\end{figure}
\end{example}
\begin{remark}
Theorem~\ref{theo-reduce2} shows that an initial circle state 
$j\in [n]$ is winning if and only
if there is a vector $x$ 
in the associated tropical polyhedral cone
\[ P:=\{x\in (\Rm)^n\mid Ax\leq Bx\}
\enspace,
\] 
such that $x_j\neq -\infty$. 
As noted in Remark~\ref{rk-dual}, all the present constructions
admit dual versions. In particular, when the action states
are finite on both sides, so that $I=[m]$, 
it can be checked that a square state $i\in [m]$ is winning for Player Min if and only if there
is a vector $y$, in the dual
tropical polyhedral cone
\[
P':=\{y\in (\R\cup\{+\infty\})^n\mid 
A^\sharp y \leq B^\sharp y
\}
\]
such that $y_i\neq +\infty$ ($P'$ is a convex cone in the min-plus
sense). To see this, it suffices
to consider the one day operator $f(y)=BA^\sharp y$, and note
that $BA^\sharp y\leq y$ holds if and only if $A^\sharp y\leq B^\sharp y$.
\end{remark}
\begin{remark}
We note that solving the system of tropical inequalities $Ax\leq Bx$
or the system of tropical equalities $ Ax=Bx$ are computationally equivalent
problems: each of the two problems trivially reduces
to the other one. Indeed, $Ax\leq Bx$ holds if and only if
$Cx= Bx$, where $C$ denotes the matrix obtained by taking the pointwise
maximum of $A$ and $B$. Conversely, $Ax=Bx$ holds if and only if 
we have both $Ax\leq Bx$ and $Bx\leq Ax$, which is again a system
of the form~\eqref{sys-trop}, but with $2m$ inequalities instead of $m$.
\end{remark}
\begin{remark}
In~\cite{atserias}, Atserias and Maneva show that 
a mean payoff game has a nonnegative mean payoff vector
if and only if an associated ``max-atom problem''~\cite{bezem2} has a
finite integer solution. The latter is equivalent to
finding a vector $u\in \mathbb{Z}^n$ such that $z\leq f(z)$.
This is analogous to the result of~\cite{dhingra} (Corollary~\ref{theo-reduce3} here), the condition that $u\in \mathbb{R}^n$ being replaced
by $u\in \mathbb{Z}^n$. This could also
be obtained as a consequence of the present Corollary~\ref{theo-reduce3},
by using Proposition~\ref{prop-integer}, showing that solvability
over the reals and over the integers are equivalent if the data are integer.
\end{remark}
\subsection{A power algorithm to check whether Player Max has a winning state}
As it is mentioned in the introduction, the problem of
computing the value of a mean payoff game with finite state
and action spaces is a well studied one, for which several 
algorithms with a fast experimental average case execution time
are known~\cite{gurvich2,gg,gg98},
although the complexity of the problem is still
unsettled~\cite{JPZSIAM}.
Zwick and Paterson~\cite[Theorem~2.3]{zwick} showed
that the value iteration allows one to determine
the value (the vector $\chi(f)$) of a mean payoff game
in pseudo-polynomial time, assuming 
that the instantaneous payments, i.e., here, the finite entries $A_{ij}$
and $B_{ij}$, are integers. 

However, to solve Problems~\ref{pb-1} or \ref{pb-indep}, we only need
to decide whether there is one winning state, and then,
the value iteration algorithm can be refined
by exploiting the Collatz-Wielandt
property,
which leads to the algorithm that we next describe.

We assume here that $I=[m]$, that Assumptions~\ref{assump1} and \ref{assump2}
hold, take $f=A\res B$, and set
\[
g(x)=\min(f(x),x) \enspace .
\]
\begin{lemma}
We have $\bar\chi(g)\geq 0$ if and only if $\bar\chi(f)\geq 0$.
\end{lemma}
\begin{proof}
Since $g\leq f$, we have $g^N(x)\leq f^N(x)$ for all $N$ and
for all $x$, and so, $\bar\chi(g)\leq \bar\chi(f)$. 
Hence, $\bar\chi(g)\geq 0$ implies $\bar\chi(f)\geq 0$.
To show the converse, let us 
take a vector $u\in (\R\cup\{-\infty\})^n$, not identically
$-\infty$, such that $f(u)= \rho(f)+u$. If 
$\bar\chi(f)\geq 0$, then, $\rho(f)=\bar\chi(f)\geq 0$, and so, $g(u)=\min(\rho(f)+u,u)= u$, 
which implies that $\bar\chi(g)=\rho(g)\geq 0$. 
\end{proof}
Observe that if $\bar\chi(g)<0$, we must have $\bar\chi(g)\leq -1/(n+m)$.
Indeed, since the mean payoff game admits optimal positional
strategies, $\rho(g)$ must be equal to the weight-to-length
ratio of an elementary circuit in the graph of the game,
and such a circuit must have a length at most $n+m$ and a negative integer
weight.

Define now the sequence 
\begin{align}
x^0=0,\qquad x^{k+1} =g(x^k) \enspace .
\end{align}

We shall make several observations.

\begin{enumerate}
\item\label{power1} If $x^k_j<0$ for all $j\in[\ncols]$, then, $\bar\chi(g)<0$.

Indeed, we have $\cw(g^k)\leq \max_{j\in[\ncols]}(x^k_j -x^0_j)<0$, and
since, by Remark~\ref{rk-iter}, $\cw(g^k)=k\cw(g)$, we deduce
that $\bar\chi(g)=\cw(g)<0$. 

\item\label{power2} If $\bar\chi(g)<0$, then $x_j^k<0$ holds for
all $j\in[\ncols]$ as soon as
 $k\geq K:=2(n+m)^2M+1$,
where $M$ is the maximal modulus
of the integers appearing as the finite coefficients of $A,B$. 

Indeed, Theorem~2.2 of~\cite{zwick} shows that
$x^k_i \leq k \chi_i(g) + 2(n+m) M$.
Hence, if $\bar\chi(g)<0$, we get 
$x^k_i\leq -k/(n+m)+ 2(n+m)M$, and
the result follows.

\item\label{power3} %
If $x^{k+1}=x^k$,
then, we readily conclude that $\bar\chi(g)=\rho(g)\geq 0$,
since $x^k\in\R^n$ (by the assumptions on $A$ and $B$, $f$ preserves
$\R^n$).

\item\label{power4} Let $J_k:=\{j\in[\ncols]\mid x^{k+1}_j <x^k_{j}\}$,
assume that $J_{k}$ is non-empty and different
from $[\ncols]$, 
and define $y^k$ by 
\[ 
y^k_j=\begin{cases} x^k_j & \text{ if } j\in [\ncols]\setminus J_k\\
-\infty & \text{otherwise}
\end{cases}
\]
By construction, $y^k$ is a non-identically $-\infty$ vector,
and it is a candidate to be a fixed point of $g$. 
If $g(y^k)=y^k$, we must have $\bar\chi(g)\geq 0$.
\end{enumerate}
The previous observations justify the following {\em power type} algorithm,
which consists in computing the sequence $x^k$. If, for some $k$, 
Condition~\eqref{power1} is satisfied, we stop the algorithm, $x^k$ being a certificate
that $\bar\chi(g)<0$. Similarly, if for some $k$, Condition~\eqref{power3} or Condition~\eqref{power4}
is satisfied, we stop the algorithm, $x^k$ or $y^k$ 
being a certificate
that $\bar\chi(g)\geq 0$. Finally, if step $k=K$ is reached,
we must have $\bar\chi(g)\geq 0$ and stop the algorithm.

This algorithm requires at most $K$ iterations, and since one
iteration takes a linear time, the algorithm is pseudo-polynomial.
However, the practical interest of testing 
Conditions~\eqref{power3} or~\eqref{power4} in addition to Condition~\eqref{power1} is that they %
can frequently be met
before the termination time $K$ is reached. A more precise
complexity analysis is beyond the scope of this paper.

\begin{remark}
The previous power algorithm can be initialized with
an arbitrary vector $x^0\in \R^n$. This reduces
to the former case if we replace $f(x)$ by $x^0+f(x-x^0)$.
Then, the payments and so the constant $M$ must be modified
accordingly, and the stopping
condition~\eqref{power1} becomes $x^k_j<x^0_j$ for all $j\in[\ncols]$.
\end{remark}

\section{Mean payoff games expressing tropical linear independence}
\subsection{Extension of the tropical semiring and linear independence}
In tropical algebra, the set of ``zeros'' of an expression
is generally defined by the requirement
that the maximum of the terms arising in this expression
is attained at least twice. 
The notation $\llq * =0\rrq$ is often used in this sense informally. This
notation can in fact be given a formal meaning, by using
an extension of the tropical semiring, which was introduced
by Izhakian~\cite{Izh}. The latter may be thought of as the ``complex''
analogue of the ``real'' (signed) extension of the tropical semiring
introduced by M. Plus~\cite{Plus}. In a nutshell, the ``numbers'' of the extension of Izhakian carry an information
reminding whether the maximum of an expression is attained twice, whereas
the ``numbers'' of the extension of M. Plus carry a sign information, reminding
whether the maximum of a signed formal expression is attained by
a positive term, by a negative one, or both.  
The approach of~\cite{Izh} has
been pursued in several works of Izhakian and Rowen like~\cite{izhrowen2}, whereas the authors
have studied in~\cite{AGG08} semirings with an abstract involution,
in order to unify both approaches. 
Such extensions provide a convenient notation, and, as shown in~\cite{Plus,AGG08}, they allow one to perform elimination arguments, as in the Gauss algorithm,
while staying at the tropical level, and to obtain automatically polynomial identities over semirings.

Although our primary interest is in the basic max-plus case, we shall
establish our results in the framework of the extended tropical semiring,
which leads to slightly more general results. The reader interested
only by the max-plus case may skip the present section,
and specialize the further sections by assuming that the
matrices have entries in the max-plus semiring rather than in its
extension.

The presentation which follows is a simplified version of~\cite{AGG08}.

\begin{definition}
Let $\N_{2}$ be the semiring which is the quotient of the semiring $\N$ of nonnegative integers by the equivalence relation which identifies all numbers greater than
or equal to $2$, and denote $\N_{2}^*=\N_2\setminus \{ 0 \}$ and 
 $\rmax^*=\rmax\setminus\{-\infty\}$.
The \NEW{extended tropical semiring} is the subset of
$\N_2\times \rmax$:
\[ \Ti:=(\N_2^*\times \rmax^*) \cup \{(0,-\infty )\}\] 
endowed with the addition
$$
(a,b)\oplus (a',b')=\left\{\begin{array}{lcl}
(a+a',b) & \text{ if } & b=b' \\
(a,b) & \text{ if } & b>b' \\
(a',b') & \text{ if } & b<b' \end{array} \right.
$$
and the multiplication
$$
(a,b)\odot (a',b')=(a \cdot a', b +  b').
$$
\end{definition}
The extended tropical semiring as defined above is a semiring with 
zero %
$\zero:=(0 ,-\infty )$ and unit $\unit:=(1,0 )$  and it is isomorphic to the 
extended tropical semiring defined in~\cite{Izh} (see~\cite{AGG08}).

The semiring $\Ti$ is not idempotent, but is ordered naturally by the relation:
$x\leq y$ if there exists $z\in\Ti$ such that $x\oplus z=y$.
The map $\proj:\Ti\to \rmax, \;(a,b)\mapsto \proj (a,b) := b$ 
is a surjective morphism, thus it is order preserving.
However the natural injection from $\rmax$ to $\Ti$,
which sends $b\in\rmax^*$ to $(1,b)$ 
and  $-\infty$ to $(0 ,-\infty )$ is not a morphism.
Nevertheless, it is a multiplicative morphism, it is order preserving 
and denoting by 
$\inj{b}$ the image of $b\in\rmax$ by this injection,
the following holds for all $x,y\in\Ti$:
\begin{equation}\label{inj-quasi-morph}
x^\vee\vee y^\vee \leq (x\oplus y)^\vee\leq x^\vee \oplus y^\vee\enspace,
\end{equation}
where $a\vee b$ denotes the least upper bound of two elements $a,b\in \Ti$.

The following notations are defined in~\cite{AGG08} for
any semiring with symmetry. We avoid here the use of the minus sign.

\begin{definition}
For any $a\in \Ti$, we set $a^\circ:=a\oplus a$, and we denote
\[
\Ti^\circ := \set{a^\circ}{a\in \Ti }  ,\qquad  \Ti^\vee:=(\Ti\setminus \Ti^\circ )\cup\{(0 ,-\infty )\}\enspace ,
\]
and we define  on $\Ti$ the \NEW{balance} relation $\balance$ by 
\[ a \balance b\iff a\oplus b\in \Ti^\circ \enspace .
\]
\end{definition}
The balance relation is reflexive, symmetric but not transitive.
Denoting $b^\circ:= (\inj b)^\circ$ for $b\in\rmax$, we get that
\[
\Ti^\circ = \set{b^\circ}{b\in \rmax }  ,\qquad  \Ti^\vee=
\set{\inj b}{b\in \rmax }\enspace .
\]
We shall say that an element of $\Ti$ is of type
{\em real} if it belongs to $\Ti^\vee$ and of type
{\em ghost} if it belongs to $\Ti^\circ$ (thus, the zero
element of the semiring has both types).
An element $a$ of $\Ti$ is determined by its projection
$\pi(a)\in \rmax$ and by its type. The elements of
$\Ti^\vee\setminus\{\zero\}$ are precisely the
invertible elements of $\Ti$.

This algebraic structure encodes whether the maximum
in an expression is attained once or at least twice.
The elements $b^\vee$ with $b\in \rmax^*$ correspond
to expressions the maximum of which is finite and is attained
only once, the elements $b^\circ$ with $b\in \rmax^*$ correspond
to expressions  the maximum of which is finite and is attained
at least twice, the element $\zero\in \Ti^\circ\cap\Ti^\vee$ corresponds
to expressions the maximum of which is $-\infty$. For instance,
the following computations are valid
\[
2^\vee \oplus 2^\vee =2^\circ, \;\; 2^\vee \oplus 3^\vee = 3^\vee,\;\;
2^\vee \oplus 3^\circ= 3^\circ \enspace .
\]

The previous notations will be extended to vectors, entrywise. For instance,
if $x,y\in \Ti^n$, we shall write $x\balance y$ if $x_j\balance y_j $
for all $j\in[\ncols]$.

\begin{definition} \label{def:trop-dependent-in-T}
If $A$ is a matrix in $\cM_{m,n}(\Ti)$, we shall say that the columns
of $A$ are tropically linearly dependent if there exists
a vector $x\in (\Ti^\vee)^n$, different from the zero vector $\zero$,
such that 
\[
Ax\balance \zero \enspace . 
\]
\end{definition}
When $A$ is the image of some matrix $B\in \mnmatrices$ by the
canonical injection, meaning that $A_{ij}=B_{ij}^\vee$, setting
$x=y^\vee$ for some $y\in \rmax^n$, we easily check that $Ax\balance \zero$
holds if and only if the maximum in each of the expressions
\[
\max_{j\in[\ncols]} (B_{ij}+y_j) \enspace , \qquad i\in[\nrows]\enspace,
\]
is attained at least twice, or equal to $-\infty$, which is
the natural notion of tropical linear dependence over $\rmax$
given in the introduction (statement of Problem~\ref{pb-indep}).

Thus, all the statements that follow which concern tropical linear
independence over $\Ti$ yield in particular %
corresponding
statements 
for tropical linear independence over $\rmax$. The interest
of the notation $Ax\balance \zero$ is its similarity with
the classical notation $Ax=0$ (the columns of a matrix
over a ring are dependent if some nontrivial
linear combination of the columns vanishes).

Tropical linear independence turns out to be controlled
by permanents.
\begin{definition} \label{de:detS}
Let $A=(A_{ij})\in\cM_{n,n}(\Ti)$.
The permanent $\per{A}$ of the matrix $A$ is the element
of $\Ti$ defined by 
$$\per{A}= \bigoplus_{\sigma\in \allperm_n} 
 A_{1\sigma(1)}\cdot\cdots\cdot A_{n\sigma(n)} ,$$
where $ \allperm_n$ denotes the set of all permutations of the set 
$[\ncols]$.
\end{definition}
If $A_{ij}=B_{ij}^\vee$ for some $B_{ij}\in \rmax$, then,
the projection onto $\rmax$ of the permanent of $A$,
\[
\proj (\per A) = \max_{\sigma\in \allperm_n} 
 (B_{1\sigma(1)}+\cdots+ B_{n\sigma(n)})
\enspace ,
\]
is the value of the optimal assignment problem with weights $B_{ij}$.
The type of $\per A$ is real if there is only one optimal
permutation, or if the value of the previous maximum
is $-\infty$, and it is ghost if there are at least two
optimal permutations. 
Moreover, $\per A$ is invertible if and only if $B$ is 
tropically nonsingular as defined in the introduction
(see Section~\ref{sec1.4}).
This suggests the following definition.
\begin{definition} \label{def:trop-nonsingular-in-T}
We shall say that the %
matrix $A\in\cM_{n,n}(\Ti)$ is
{\em tropically nonsingular} if $\per A$ is invertible in $\Ti$.
\end{definition}

In the sequel, we shall establish results %
for matrices with entries in the extended tropical semiring $\Ti$. 
Then, we shall derive the analogous results for matrices
with entries in the tropical semiring as immediate corollaries.

\subsection{Reducing tropical linear independence to mean payoff games}\label{sec-linindep}

We denote by $A$ an $m \times n$ matrix with entries in $\Ti$,
and we shall assume:
\begin{assump}\label{assumpacol}
The matrix $A$ has no column %
consisting only of elements of $\Ti^\circ$. 
\end{assump}
This assumption
is not restrictive%
, for if $A$ had such a 
column, the columns of $A$ were tropically linearly dependent,
and when $m\geq n$, all $n\times n$ submatrices were tropically singular,
so that the equivalence 
which we are going to prove
in Theorem~\ref{th-main} 
for matrices satisfying Assumption~\ref{assumpacol}
is trivially 
true in this situation.

We set
\begin{subequations}
\begin{equation}
E=\set{(i,j)}{A_{ij}\in \Ti^\vee\setminus\{\zero\}} \enspace .
\end{equation}
Thanks to Assumption~\ref{assumpacol}, for all $j\in[\ncols]$, there
is at least one index $i\in[\nrows]$ such that $(i,j)\in E$.

We define the min-max function
$f: (\R\cup\{-\infty\})^n\to (\R\cup\{-\infty\})^n$ given by
\begin{equation}\label{def-f-dep}
f_j(x)=\min_{i\in[\nrows], \; (i,j)\in E} (-  B_{ij}+\max_{k\in[\ncols],\; k\neq j} %
(B_{ik} + x_k)) \enspace ,
\end{equation}
where
\begin{equation}\label{B=piA}
B_{ij}:=\proj A_{ij}\in \rmax \enspace .
\end{equation}
\label{fandB}
\end{subequations}
This function can be interpreted as the dynamic programming operator
of the following combinatorial game, which is played on
a bipartite digraph with $n$ column nodes and $m$ row nodes.
Being
in a column node $j$, player Min chooses a row node $i$ such that 
$(i,j)\in E$, and moves to node $i$ receiving $B_{ij}$. Then, player Max must
move to some column node $k$ which is different from the previously
visited column node $j$, and he receives $B_{ik}$. Thus, 
when all entries of $A$ are in  $\Ti^\vee$ (that is $A=B^\vee$),
player Min is advantaged, because she can always come back
to the state from which player Max just came, ensuring her a $0$ loss.
In that case, it follows that $\bar\chi(f) \leq 0$.

Such a game
may be put in the form studied in Section~\ref{sec-gen},
in which the available actions only depend on the current state,
by adding the previously visited node to the state. Formally,
the map $f$ may be written as $f(x)=C\res Dx$, where
$C,D$ are $(mn)\times  n$ matrices, with 
\begin{equation}\label{defCD}
 C_{(i,j),k}=\begin{cases}B_{ij} & \text{ if }k= j\text{ and }(i,j)\in E\\
-\infty & \text{ otherwise,}
\end{cases}
\qquad D_{(i,j),k}=\begin{cases}B_{ik} & \text{ if }k\neq j\\
-\infty & \text{ otherwise.}
\end{cases}
\end{equation}
Due to Assumption~\ref{assumpacol}, no column of $C$ is identically
$-\infty$, hence $f$ sends $(\R\cup\{-\infty\})^n$ to itself.
However, some rows of $D$ may be identically $-\infty$, 
as soon as $A$ has a row with at most one
element not equal to $-\infty$.
In that case the map $f$ may not send $\R^n$ to itself.
But one can apply Proposition~\ref{exist-chibar},
Lemma~\ref{lemma-allequal} and Theorem~\ref{theo-reduce}.

\begin{theorem}\label{th:u<f(u)} %
Let $A$ be an $m \times n$ matrix with entries in $\Ti$, satisfying 
Assumption~\ref{assumpacol}.
Let $E$, $B$ and $f$ be defined as in~\eqref{fandB}. Then
a vector $u\in \rmax^n$ is such that 
\[
Au^\vee \balance \zero
\]
if and only if $u\leq f(u)$.
\end{theorem}
\begin{proof}
If $A\,\inj{u}\notbalance\zero$, then, there exists some index
$i$ such that $(A\,\inj{u})_i=\bigoplus_j A_{ij}\odot \inj{u}_j$ is invertible in $\Ti$, which implies that there exists some index $j$ such that $A_{ij}$ is invertible and $\proj A_{ij}+ u_j >\max_{k\in[\ncols],\; k\neq j}(\proj A_{ik}+u_k)$. With the above definitions of $B$ and $E$, it follows that $(i,j)\in E$
and $u_j>-B_{ij} + \max_{k\in[\ncols],\; k\neq j}  (B_{ik}+u_k)$. 
We deduce that 
$u_j >\min_{i\in[\nrows],\; (i,j)\in E} (-B_{ij} + \max_{k\in[\ncols],\; k\neq j} (B_{ik}+u_k))=f_j(u)$. The previous deductions turn out to be equivalences,
and so, $A\,\inj{u} \notbalance\zero$ iff there exists some index $j$ such that 
$u_j> f_j(u)$. By negating both conditions, this shows the theorem.
\end{proof}

We get as an immediate consequence.
\begin{corollary} \label{c:u<f(u)}
Let $B$ be an $m \times n$ matrix with entries in $\rmax$  which has no column consisting only of elements~$-\infty$. %
Denote $E=\set{(i,j)}{B_{ij}\neq -\infty }$, and define $f$ 
by~\eqref{def-f-dep}.
Let $u$ be a vector in $(\Rm)^n$, not identically $-\infty$.
Then, the following conditions are equivalent:
\begin{enumerate}
\item
$u\leq f(u)$;
\item The equation $\llq Bu =0\rrq$ holds in the tropical sense,
meaning that in every expression
\[
\max_{j\in [n]}(B_{ij}+u_j),\qquad i\in [m]
\]
the maximum is attained at least twice or is equal to $-\infty$;
\item All the rows of the matrix $B$ are contained in the tropical hyperplane consisting of those vectors $x\in \rmax^n$ such that the maximum in $\max_{j\in[n]} (x_j+u_j)$ is attained at least twice or is equal to $-\infty$. \qed
\end{enumerate}
\end{corollary}
The following theorem %
provides an expression of
tropical linear independence
in terms of mean payoff games.
\begin{theorem}\label{th:4equivmaxplus}
Let $B$ be an $m \times n$ matrix with entries in $\rmax$  which has no column consisting only of elements~$-\infty$.  %
Denote $E=\set{(i,j)}{B_{ij}\neq -\infty }$, and define $f$ 
by~\eqref{def-f-dep}. 
The following assertions are equivalent.
\begin{enumerate}
\item\label{th:4equiv-1} 
The columns of the matrix $B$ are tropically independent;
\item \label{th:4equiv-2} Player Max has no winning state in the mean payoff game with dynamic programming operator $f$, i.e., $\bar\chi(f)<0$;
\item \label{th:4equiv-3}
 there exists a vector $w\in \R^n$ and a scalar $\lambda<0$ such that
\[
f(w)\leq \lambda +w \enspace ;
\]
\item \label{th:4equiv-4} there is no vector $u\in (\Rm)^n$, 
without finite entries,
such that $u\leq f(u)$ \enspace . %
\end{enumerate}
\end{theorem}
In fact, we shall prove the following more general result, in the
setting of the extended tropical semiring.

\begin{theorem}\label{th:4equiv}
Let $A$ be an $m \times n$ matrix with entries in $\Ti$, satisfying 
Assumption~\ref{assumpacol}.
Let $E$, $B$ and $f$ be defined as in~\eqref{fandB}. Then,
the columns of the matrix $A$ are tropically linearly independent
if and only if the map $f$ satisfies one of the three equivalent
conditions~{\rm(\ref{th:4equiv-2},\ref{th:4equiv-3},\ref{th:4equiv-4})}
of Theorem~\ref{th:4equivmaxplus}. 

\end{theorem}
\begin{proof}
By Definition~\ref{def:trop-dependent-in-T} and Theorem~\ref{th:u<f(u)}, 
the columns of $A$ are tropically dependent
if and only if there exists a vector $u\in (\Rm)^n$ non-identically $-\infty$ 
such that $u\leq f(u)$.
This shows that Property~\eqref{th:4equiv-4} is equivalent
to the tropical linear independence of the columns of $A$.
By definition, \eqref{th:4equiv-3} is equivalent to $\cw(f)<0$ and
\eqref{th:4equiv-4} is equivalent to $\cw'(f)\not\geq 0$.
From Assumption~\ref{assumpacol}, the matrix $C$ of~\eqref{defCD} has no 
column identically equal to $-\infty$, hence 
the equivalence between~\eqref{th:4equiv-2}, \eqref{th:4equiv-3}
and~\eqref{th:4equiv-4} follows from Lemma~\ref{lemma-allequal}.
\end{proof}

\subsection{Characterizations of the tropical rank}
We shall now derive Theorem~\ref{theorem2} and related results concerning rank of matrices in the more general framework of matrices with entries in $\Ti$.
We shall need the following tropical Cramer theorem proved in~\cite{AGG08}
which is analogous to the Cramer theorem of M. Plus~\cite{Plus}, the symmetrized
max-plus semiring being now replaced by the extended tropical semiring.
This is a refinement of a result concerning
the tropical Cramer formula  %
 stated by Richter-Gebert, Sturmfels, and
Theobald in~\cite{RGST}, which deals with a generic case.
\begin{theorem}[{\cite[Theorem 6.6]{AGG08}}]\label{th-cramerI}
Let $A\in {\mathcal{M}}_n(\Ti)$ and $b\in \Ti^n$, then 
\begin{enumerate}
\item Every real solution $x$ of the linear system
\[%
Ax\balance b \]%
satisfies the relation
$ (\detp A) x \balance A\adj  b$.
\item \label{th-cramerI-point2}
Moreover, if the vector $A\adj  b$ is real and $\detp A$ is invertible in $\Ti$, then 
$\hat x:={\detp A}^{-1} A\adj b$
is the unique real solution of $Ax\balance b$.
\end{enumerate}
Here $A\adj$ is defined by $(A\adj)_{ji}=\per A (i|j)$, where $A(i|j)$ is the matrix obtained from $A$ by deleting $j$-th column and $i$-th row.
\end{theorem}
As a consequence of the game formulation and of the latter theorem, we
obtain the following result.
\begin{theorem}\label{th-main}
Let $A\in \cM_{m,n}(\Ti)$ with $m\geq n$. Then, the columns of $A$ are tropically linearly
independent if and only if $A$ has an $n \times n$ submatrix that is tropically nonsingular.
\end{theorem}
Izhakian and Rowen obtained independently the same result in a recent work~\cite{IRowen}, by a different method.
\begin{proof}
The ``if'' part of the assertion was already shown in~\cite[Lemma 8.1]{AGG08}
using Theorem~\ref{th-cramerI}. We reproduce the proof for completeness.
Assume that $A$ has a tropically nonsingular submatrix of maximal size,
denote it by $F$, and assume by contradiction
that $Ax\balance \zero$ for some real non-zero vector $x$. Then,
$Fx\balance \zero$. Theorem~\ref{th-cramerI} implies that $(\per F) x\balance F\adj \zero$. Since $F\adj \zero=\zero$ it follows that $(\per F) x\balance \zero $.
Since at least one coordinate of $x$ is non-zero, 
and $x$ is real, this coordinate is invertible, hence $\per F\balance \zero$, 
which contradicts the assumption.

Let us show the ``only if'' part.
Assume that the columns of $A$ are tropically linearly independent.
This implies in particular that no column of $A$ consists of elements of 
 $\Ti^\circ$, that is Assumption~\ref{assumpacol} is satisfied.
Then, by Theorem~\ref{th:4equiv}, $f$ satisfies Condition~\eqref{th:4equiv-3} 
of Theorem~\ref{th:4equivmaxplus}, that is there exists a vector $u\in \R^n$
and a scalar $\lambda<0$ such that
\[
\min_{i\in[\nrows], \; (i,j)\in E} (-\proj A_{ij}+\max_{k\in[\ncols],\; k\neq j} ( \proj A_{ik} + u_k)) \leq \lambda +u_j ,\;\; j\in[\ncols] \enspace .
\]
Hence, for every $j\in[\ncols]$, we can find an index $\sigma(j)$ which attains the minimum, and so $(\sigma(j),j)\in E$ and
\[
-\proj A_{\sigma(j)j}-u_j + \max_{k\in[\ncols],\; k\neq j} (\proj A_{\sigma(j)k} + u_k )\leq \lambda \enspace .
\]
Let 
$G_{ij}:= \proj A_{ij}+u_j$. 
The latter inequality can be rewritten as
\[
G_{\sigma(j)k}- G_{\sigma(j)j} \leq \lambda \enspace ,\quad j\in[\ncols],\;
\quad k\in[\ncols],\; k\neq j \enspace .
\]
We claim that $\sigma$ is injective. Indeed, assume by contradiction
that $i=\sigma(j')=\sigma(j'')$ for some $j'\neq j''$. Then, by selecting $k=j'$ and $j=j''$ in the previous inequality, we get
\[
G_{ij'}-G_{ij''}\leq \lambda \enspace .
\]
and by selecting $k=j''$ and $j=j'$ we get symmetrically, 
\[
G_{ij''}-G_{ij'}\leq \lambda \enspace.
\]
Summing these inequalities, we get $0\leq 2 \lambda<0$, which is nonsense.

Let $I:=\set{\sigma(j)}{j\in[\ncols]}$. We get that the submatrix 
consisting of the rows of $G$ with indices in $I$ is such that the maximum of row $\sigma(j)$
is attained only at column $j$. It follows that $\sigma^{-1}$ determines
the unique optimal solution of the optimal assignment problem
corresponding to this submatrix, that is
\[ \sum_{i\in I} G_{i\sigma^{-1}(i)}>\sum_{i\in I} G_{i\tau(i)}\]
for 
all
 bijective maps $\tau:I\to[\ncols]$.
Hence the same holds for the matrix $B=\proj A$ instead of $G$.
Since $(i,\sigma^{-1}(i))\in E$ for all $i\in I$, the submatrix $F$ 
consisting of the rows of $A$ with indices in $I$ satisfies 
$\per F=\per (\pi F)^\vee \in  \Ti^\vee\setminus\{\zero\}$, that is, $F$ is tropically nonsingular.
\end{proof}

The classical Radon theorem shows that $m+1$ vectors in dimension $m$
can be partitioned in two subsets that generate two convex cones
with a non-zero intersection. Tropical versions of this result
have appeared in~\cite{But,BriecHorvath04,AG,GM08,AGG08}. 
The following result is another Radon analogue, this times
in the semiring $\Ti$. 
The proof uses the same method as the one of~\cite{AGG08}.
\begin{corollary}\label{n+1dep}
Any $m+1$ vectors of $\Ti^m$ are tropically linearly dependent.
\end{corollary}
\begin{proof}
Consider $v_1,\ldots , v_{m+1}$ in $\Ti^m$ and denote by $B$ the
$m\times (m+1)$ matrix with columns $v_j$, $j\in[m+1]$.
If one of the $m\times m$ submatrices of $B$ is tropically singular, 
Theorem~\ref{th-main} shows that the columns of this submatrix are 
tropically linearly 
dependent, hence the $m+1$ vectors $v_1,\ldots, v_{m+1}$ are also 
tropically linearly dependent.
Otherwise, if $A$ is the square submatrix obtained from $B$ by deleting 
the last column, and $b=v_{m+1}$, we get that $\detp A$ is invertible in
$\Ti$ and that the vector
$\hat x:={\detp A}^{-1} A\adj b$  has all its entries in $\Ti^\vee$
(and non zero).
Hence, by Point~\eqref{th-cramerI-point2} of Theorem~\ref{th-cramerI},
$\hat x$ is a solution of $Ax\balance b$.
Denoting by $y$ the $m+1$ dimensional vector 
obtained by completing $x$ with an $m +1$-entry equal to $\unit$, we get that
$By\in \Ti^\circ$. Since $y$ is a non-zero real vector, this 
shows that the columns of $B$ are tropically linearly dependent.
\end{proof}
\begin{corollary}\label{rank-equiv}
  Let $A\in \cM_{m,n}(\Ti)$. Then, the maximal number of tropically linearly independent rows of $A$, the maximal number of tropically linearly independent columns of $A$, and the maximal size %
  of a tropically nonsingular submatrix of $A$
coincide.
\end{corollary}
\begin{proof}
Assume that $A$ has a tropically nonsingular submatrix %
of size $k$.
Let $F$ denote the submatrix of $A$ 
formed by
the columns
of $A$ corresponding to the columns of this submatrix. It follows from
the ``if'' part of Theorem~\ref{th-main} that the columns of $F$
are tropically linearly independent. Hence, $A$ contains
at least $k$ tropically independent columns,
which shows that the maximal number of tropically linearly independent columns of $A$ is greater than or equal to the maximal size 
of a tropically nonsingular submatrix of $A$.

Conversely, assume that $A$ has $k$ 
tropically independent columns. Then Corollary~\ref{n+1dep} shows that
$k\leq m$. Let $F$ denote the submatrix
of $A$ consisting of these columns. By the ``only if'' part of
Theorem~\ref{th-main},
we can find a $k\times k$ submatrix of $F$ which is tropically nonsingular,
which shows the reverse inequality, thus the
equality between  the maximal number of tropically linearly independent columns of $A$ and the maximal size of a tropically nonsingular submatrix of $A$.

Replacing $A$ by its transpose matrix $A^t$, we obtain the same result for rows instead of columns, which finishes the proof of the corollary.
\end{proof}

We next give several corollaries of these results for $\rmax$. 
Till the end of this section we
shall consider tropical linear dependence and tropical nonsingularity in $\rmax$, i.e., in the sense given in the introduction.

\begin{corollary}\label{th:main}
Let $A\in \cM_{m,n}(\rmax)$ with $m\geq n$. Then, the columns of $A$ are tropically linearly
independent if and only if $A$ has an $n \times n$ submatrix, which is tropically nonsingular.
\end{corollary}

Recall that the {\em tropical rank} of a matrix $A\in \cM_{m,n}(\rmax)$
is defined as the maximal size of a tropically non-singular submatrix.
In~\cite{AGG08}, we also defined the \NEW{maximal row (resp.\ column) rank} 
of a matrix $A$ with entries in $\rmax$ as 
the maximal number of tropically linearly independent rows (resp.\ columns)
of $A$.
We get as an immediate corollary of Corollary~\ref{rank-equiv}
the equivalence between all these rank notions.

\begin{corollary}
  Let $A\in \cM_{m,n}(\rmax)$. Then, the maximal row rank of $A$,
 the maximal column rank of $A$ and the tropical rank of $A$ coincide.
\end{corollary}
\begin{corollary}\label{cor-trivial}
Checking whether a matrix $A\in \cM_{m,n}(\rmax)$, with $m\geq n$, has
tropical rank at least $n-k$, reduces to solving ${ n\choose k}$ mean
payoff game problems associated to $m\times (n-k)$ matrices,
and can therefore be done in pseudo-polynomial time
for a fixed value of $k$.
\end{corollary}
\begin{proof}
It suffices to check, for every subset $I$ of $[n]$ of cardinality
$n-k$, whether the columns in $I$ of $A$ are tropically linearly independent,
which, by Theorem~\ref{th:4equivmaxplus}, is a mean payoff game problem.
\end{proof}
\begin{remark}\label{rk-trivial}
Recall that checking whether a square matrix is tropically
singular can be done in $O(n^3)$ time, as observed by Butkovi\v{c}~\cite{butkovip94}. Hence, for a fixed $k$, checking whether a matrix
has tropical rank strictly less than $k$ is also a polynomial time problem
(it suffices to check whether all the ${m \choose k}\times {n\choose k}$ submatrices of size $k\times k$ are tropically singular). However, Kim and Roush
showed that the general problem of computing the tropical rank is NP-hard~\cite{kim}. 
\end{remark}

\begin{example}
Consider the  points in $\rmax^3$,
\[
a=\begin{pmatrix}
0 & 2 & 0\end{pmatrix}
\quad
b=\begin{pmatrix}
0 & 3 & 2\end{pmatrix}
\quad
c=\begin{pmatrix}
0 & 1& 1 \end{pmatrix}
\quad
d=\begin{pmatrix}
1& 3& 0 \end{pmatrix}
\quad
e=\begin{pmatrix}
1 & 1 & 0
\end{pmatrix}
\]
These points are represented in Figure~\ref{fig-hyper}.
\begin{figure}[htbp]
\[
\input{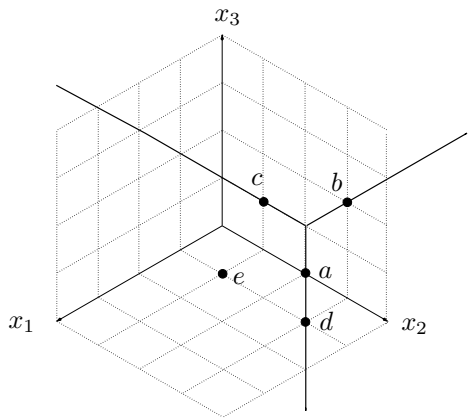}
\]
\caption{Tropical hyperplane passing through the points $a,b,c,d$}
\label{fig-hyper}
\end{figure}
The hyperplane $H$ defined by the condition that
the maximum in the expression
\[
\max(2+x_1,x_2,1+x_3)
\]
is attained at least twice is represented by the union
of three bold half-lines. %
The points $a,b,c,d$ belong to this hyperplane, but it is easy to check graphically there is no hyperplane containing the five points $a,b,c,d,e$. 

We next show that these conclusions can be obtained by the previous arguments.
The game associated to the matrix with rows $a,b,c,d$ can be visualized in
Figure~\ref{fig-gamehyper}.
\begin{figure}[htbp]
\[
\input{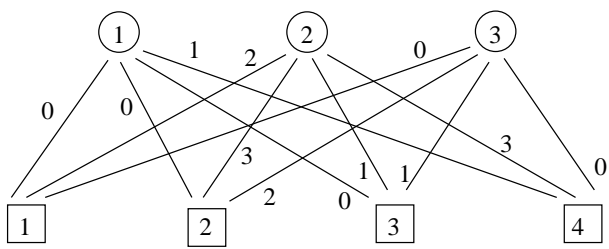}
\]
\caption{Game associated with the configurations of points in Figure~\ref{fig-hyper}}\label{fig-gamehyper}
\end{figure}
Recall that in the square states, Player Max must choose a new circle state 
different from the previously visited one.
The dynamic programming operator $f$ has the three following coordinates:
{\small \[
\begin{array}{l}
f_1(x):=\min(\max(2+x_2,x_3),\max(3+x_2,2+x_3),\max(1+x_2,1+x_3),-1+\max(3+x_2,x_3))\\
f_2(x):=\min(-2+\max(x_1,x_3),-3+\max(x_1,2+x_3),-1+\max(x_1,1+x_3),-3+\max(1+x_1,x_3))
\\
f_3(x):=\min(\max(x_1,2+x_2),-2+\max(x_1,3+x_2),-1+\max(x_1,1+x_2),\max(1+x_1,3+x_2)).
\end{array}
\]}
By applying the power algorithm
to the map $g(x)=\min(f(x),x)$, starting
from $x^0=(0,0,0)$, we compute $x^1:=g(x^0)=(0,-2,0)$, and $x^2:=g(x^1)=
(0,-2,-1)$, which is a fixed point of $g$, and so the algorithm stops.
The coefficients
of the fixed point $(0,-2,-1)$ determine the half-space $H$ containing
$a,b,c,d$. 

Let us now replace the vector $d$ by the vector $e$. The map $f$ becomes
{\small
\[
\begin{array}{l}
f_1(x):=\min(\max(2+x_2,x_3),\max(3+x_2,2+x_3),\max(1+x_2,1+x_3),-1+\max(1+x_2,x_3))\\
f_2(x):=\min(-2+\max(x_1,x_3),-3+\max(x_1,2+x_3),-1+\max(x_1,1+x_3),-1+\max(1+x_1,x_3))
\\
f_3(x):=\min(\max(x_1,2+x_2),-2+\max(x_1,3+x_2),-1+\max(x_1,1+x_2),\max(1+x_1,1+x_2)).
\end{array}
\]}
The power algorithm, with the same initial condition, gives
$x^1=(0,-2,0)$, $x^2=(-1,-2,-1)$, and since $x^2_i<0$ for all $i=1,2,3$,
 the power
algorithm stops, showing that the vectors $a,b,d,e$ are tropically
linearly independent. 
\end{example}
\subsection{Alternative proof of Theorem~\ref{th-main} via the tropical Helly theorem and further comments}\label{sec-matroid}
When $m=n$, Theorem~\ref{th-main} can be proved by direct
combinatorial means (essentially by network flows arguments)
as is done in~\cite{Izh}. 
We next observe that the $m>n$ case can
be reduced to the $m=n$ case by means
of the tropical Helly theorem, which appeared
in the works of Briec and Horvath~\cite{BriecHorvath04},
Gaubert and Sergeev~\cite{gauser}, and Gaubert and Meunier~\cite{GM08}, 
with three different proofs.

\begin{theorem}[Tropical Helly theorem, {\cite{BriecHorvath04,gauser,GM08}}]
Let $(C_i)_{i\in[\nrows]}$ denote a collection of
tropical cones of $\rmax^n$. %
If the intersection of all the $C_i$
is reduced to the zero vector, then there exists a subcollection of cardinality
$n$ the intersection of which is also reduced to the zero vector.
\end{theorem}

The following corollary shows that the rectangular case ($m>n$) in Theorem~\ref{th-main} can be derived from the square case ($m=n$).

\begin{corollary}
Let $A\in \cM_{m,n}(\Ti)$ with $m\geq n$ has tropical linearly independent
columns, then, it has an $n \times n$ submatrix the columns of which are still
tropically linearly independent.
\end{corollary}
\begin{proof}
We apply the tropical Helly theorem to 
\[
C_i = \set{x\in \rmax^n}{A_i \inj{x} \balance \zero}
\]
where $A_i$ denotes the $i$th row of $A$. The set $C_i$ is 
a tropical cone. 
Since the map $x\mapsto \inj{x}$ is not a morphism of semirings
this is not immediate.
But it is a multiplicative morphism, %
hence $C_i$ is stable by tropical
multiplication: $x\in C_i$ and $\lambda\in\rmax$ imply $\lambda x\in C_i$ using
$\inj{(\lambda x)}=\inj{\lambda}\inj{x}$.

To prove that $C_i$ is stable by tropical sum, we shall use Property~\eqref{inj-quasi-morph} on page~\pageref{inj-quasi-morph}
for vectors with entries in $\Ti$. Indeed,
as for $\Ti$, $\Ti^n$ can be endowed with its natural order,
$u\leq v$ if there exists $w\in\Ti^n$ such that  $u\oplus w=v$, 
which coincide with the pointwise order.
Hence, since Property~\eqref{inj-quasi-morph} holds for scalars in $\Ti$, it 
also holds for vectors with entries in $\Ti$.
Let $x,y\in C_i$, then $A_i \inj{x}$ and $A_i \inj{y}$ are in
$\Ti^\circ$. Denote by $z=x\oplus y$ the tropical sum of $x$ and $y$.
By Property~\eqref{inj-quasi-morph} for $x$ and $y$, we get that
$x^\vee\leq z^\vee\leq x^\vee\oplus y^\vee$.
Applying $A_i$, which is additive thus order preserving on $\Ti^n$,
we obtain
$A_ix^\vee\leq A_iz^\vee\leq A_i(x^\vee\oplus y^\vee)=A_i x^\vee\oplus A_i y^\vee$.
Replacing $x$ by $y$ and taking the supremum of both inequalities, we get 
\[ A_ix^\vee\vee A_i y^\vee \leq A_iz^\vee\leq A_i x^\vee\oplus A_i y^\vee\enspace .\]
Since $A_i \inj{x}$ and $A_i \inj{y}$ are in
$\Ti^\circ$, we deduce that $A_ix^\vee\vee A_i y^\vee=A_i x^\vee\oplus A_i y^\vee$, which together with the previous inequality implies that 
$A_i z^\vee=A_i x^\vee\oplus A_i y^\vee\in\Ti^\circ$.
This shows that $z\in C_i$, so that $C_i$ is a tropical cone.

By definition, the columns of the matrix $A$
are linearly independent if and only if the intersection of all the $C_i$
is reduced to the zero vector. By the tropical Helly theorem,
we can find $i_1<i_2<\cdots<i_n$ such that $C_{i_1}\cap \cdots \cap C_{i_n}=\{\zero\}$. It follows that the columns of the submatrix $F$ consisting of the rows $i_1,\ldots,i_n$ of $A$ are tropically linearly independent.
\end{proof}

\begin{remark}
The difficulty of computing the tropical
rank is related to the lack of matroid structure, see~\cite[\S~7]{DSS}.
\end{remark}

\paragraph{Acknowledgements}
The first two authors thank Melody Chan for having pointed out to us that the special case
of Theorem~\ref{theorem2} in which the entries of the matrix are finite can
be deduced from Theorem~5.5 of~\cite{DSS}. The second author thanks
Jack Edmonds for having brought reference~\cite{skutella} to his attention.

The paper was written when the third author was visiting the Maxplus team at INRIA, Paris - Rocquencourt, and CMAP (\'Ecole Polytechnique and INRIA, Saclay - \^Ile-de-France).  He would like to thank the colleagues from both institutions for their warm hospitality.
\bibliographystyle{alpha}
\bibliography{tropical}
\end{document}